\documentclass[12pt]{amsart}
\usepackage{amssymb,latexsym}
\usepackage{epsfig}
\usepackage[usenames,dvipsnames]{pstricks}
\usepackage[utf8]{inputenc}
\title[Boij-S\"oderberg theory]{Boij-S\"oderberg theory: Introduction 
and survey}
\begin{document}

\author{Gunnar Fl{\o}ystad}
\address{Matematisk Institutt\\
         Johs. Brunsgt. 12\\
         5008 Bergen}
\email{gunnar@mi.uib.no}



\begin{abstract}
{Boij-S\"oderberg theory describes the Betti diagrams of graded modules
over the polynomial ring, up to multiplication by a rational number.
Analog Eisenbud-Schreyer theory describes the cohomology tables
of vector  bundles on projective spaces up to rational multiple.
We give an introduction and survey of these newly developed areas.}
\end{abstract}

\maketitle
\tableofcontents

\theoremstyle{plain}
\newtheorem{theorem}{Theorem}[section]
\newtheorem{corollary}[theorem]{Corollary}
\newtheorem*{main}{Main Theorem}
\newtheorem{lemma}[theorem]{Lemma}
\newtheorem{proposition}[theorem]{Proposition}
\newtheorem{conjecture}[theorem]{Conjecture}
\newtheorem*{theoremA}{Theorem}
\newtheorem*{theoremB}{Theorem}
\newtheorem*{pieri}{Pieri's rule}
\theoremstyle{definition}
\newtheorem{definition}[theorem]{Definition}
\newtheorem{question}[theorem]{Question}
\newtheorem{problem}{Problem}
\theoremstyle{remark}
\newtheorem{notation}[theorem]{Notation}
\newtheorem{remark}[theorem]{Remark}
\newtheorem{example}[theorem]{Example}
\newtheorem{claim}{Claim}


\newcommand{\psp}[1]{{{\bf P}^{#1}}}
\newcommand{\psr}[1]{{\bf P}(#1)}
\newcommand{\opw}{\op_{\psr{W}}}
\newcommand{\go}{\op}

\newcommand{\ini}[1]{\text{in}(#1)}
\newcommand{\gin}[1]{\text{gin}(#1)}
\newcommand{\kr}{{\Bbbk}}
\newcommand{\kk}{{\Bbbk}}
\newcommand{\pd}{\partial}
\newcommand{\vardel}{\partial}
\renewcommand{\tt}{{\bf t}}


\newcommand{\coh}{{{\text{{\rm coh}}}}}


\newcommand{\modv}[1]{{#1}\text{-{mod}}}
\newcommand{\modstab}[1]{{#1}-\underline{\text{mod}}}

\newcommand{\sut}{{}^{\tau}}
\newcommand{\sumit}{{}^{-\tau}}
\newcommand{\til}{\thicksim}

\newcommand{\totp}{\text{Tot}^{\prod}}
\newcommand{\dsum}{\bigoplus}
\newcommand{\dprod}{\prod}
\newcommand{\lsum}{\oplus}
\newcommand{\lprod}{\Pi}

\newcommand{\La}{{\Lambda}}
\newcommand{\lam}{{\lambda}}
\newcommand{\GL}{{GL}}

\newcommand{\sirstj}{\circledast}

\newcommand{\she}{\EuScript{S}\text{h}}
\newcommand{\cm}{\EuScript{CM}}
\newcommand{\cmd}{\EuScript{CM}^\dagger}
\newcommand{\cmri}{\EuScript{CM}^\circ}
\newcommand{\cler}{\EuScript{CL}}
\newcommand{\clerd}{\EuScript{CL}^\dagger}
\newcommand{\clerri}{\EuScript{CL}^\circ}
\newcommand{\gor}{\EuScript{G}}
\newcommand{\gF}{\mathcal{F}}
\newcommand{\gG}{\mathcal{G}}
\newcommand{\gM}{\mathcal{M}}
\newcommand{\gE}{\mathcal{E}}
\newcommand{\gD}{\mathcal{D}}
\newcommand{\gI}{\mathcal{I}}
\newcommand{\gP}{\mathcal{P}}
\newcommand{\gK}{\mathcal{K}}
\newcommand{\gL}{\mathcal{L}}
\newcommand{\gS}{\mathcal{S}}
\newcommand{\gC}{\mathcal{C}}
\newcommand{\gO}{\mathcal{O}}
\newcommand{\gJ}{\mathcal{J}}
\newcommand{\gU}{\mathcal{U}}
\newcommand{\mm}{\mathfrak{m}}

\newcommand{\dlim} {\varinjlim}
\newcommand{\ilim} {\varprojlim}

\newcommand{\CM}{\text{CM}}
\newcommand{\Mon}{\text{Mon}}


\newcommand{\Kom}{\text{Kom}}


\newcommand{\EH}{{\mathbf H}}
\newcommand{\res}{\text{res}}
\newcommand{\Hom}{\text{Hom}}
\newcommand{\inhom}{{\underline{\text{Hom}}}}
\newcommand{\Ext}{\text{Ext}}
\newcommand{\Tor}{\text{Tor}}
\newcommand{\ghom}{\mathcal{H}om}
\newcommand{\gext}{\mathcal{E}xt}
\newcommand{\id}{\text{{id}}}
\newcommand{\im}{\text{im}\,}
\newcommand{\codim} {\text{codim}\,}
\newcommand{\resol}{\text{resol}\,}
\newcommand{\rank}{\text{rank}\,}
\newcommand{\lpd}{\text{lpd}\,}
\newcommand{\coker}{\text{coker}\,}
\newcommand{\supp}{\text{supp}\,}
\newcommand{\Ad}{A_\cdot}
\newcommand{\Bd}{B_\cdot}
\newcommand{\Gd}{G_\cdot}


\newcommand{\sus}{\subseteq}
\newcommand{\sups}{\supseteq}
\newcommand{\pil}{\rightarrow}
\newcommand{\vpil}{\leftarrow}
\newcommand{\lvpil}{\longleftarrow}
\newcommand{\rpil}{\leftarrow}
\newcommand{\lpil}{\longrightarrow}
\newcommand{\inpil}{\hookrightarrow}
\newcommand{\pils}{\twoheadrightarrow}
\newcommand{\projpil}{\dashrightarrow}
\newcommand{\dotpil}{\dashrightarrow}
\newcommand{\adj}[2]{\overset{#1}{\underset{#2}{\rightleftarrows}}}
\newcommand{\mto}[1]{\stackrel{#1}\longrightarrow}
\newcommand{\vmto}[1]{\overset{\tiny{#1}}{\longleftarrow}}
\newcommand{\mtoelm}[1]{\stackrel{#1}\mapsto}

\newcommand{\eqv}{\Leftrightarrow}
\newcommand{\impl}{\Rightarrow}

\newcommand{\iso}{\cong}
\newcommand{\te}{\otimes}
\newcommand{\tek}{\te_\kr}
\newcommand{\sqte}{\te}
\newcommand{\into}[1]{\hookrightarrow{#1}}
\newcommand{\ekv}{\Leftrightarrow}
\newcommand{\equi}{\simeq}
\newcommand{\isopil}{\overset{\cong}{\lpil}}
\newcommand{\equipil}{\overset{\equi}{\lpil}}
\newcommand{\ispil}{\isopil}
\newcommand{\vvi}{\langle}
\newcommand{\hvi}{\rangle}
\newcommand{\susneq}{\subsetneq}
\newcommand{\sgn}{\text{sign}}


\newcommand{\xd}{\check{x}}
\newcommand{\ortog}{\bot}
\newcommand{\tL}{\tilde{L}}
\newcommand{\tM}{\tilde{M}}
\newcommand{\tH}{\tilde{H}}
\newcommand{\tvH}{\widetilde{H}}
\newcommand{\tvh}{\widetilde{h}}
\newcommand{\tV}{\tilde{V}}
\newcommand{\tS}{\tilde{S}}
\newcommand{\tT}{\tilde{T}}
\newcommand{\tR}{\tilde{R}}
\newcommand{\tf}{\tilde{f}}
\newcommand{\ts}{\tilde{s}}
\newcommand{\tp}{\tilde{p}}
\newcommand{\tr}{\tilde{r}}
\newcommand{\tfst}{\tilde{f}_*}
\newcommand{\empt}{\emptyset}
\newcommand{\bfa}{{\bf a}}
\newcommand{\bfb}{{\bf b}}
\newcommand{\bfd}{{\bf d}}
\newcommand{\bfe}{{\bf e}}
\newcommand{\bfp}{{\bf p}}
\newcommand{\bfc}{{\bf c}}
\newcommand{\bfl}{{\bf \ell}}
\newcommand{\bfz}{{\bf z}}
\newcommand{\bfj}{{\bf j}}
\newcommand{\ubfd}{\underline{\bfd}}
\newcommand{\la}{\lambda}
\newcommand{\bfen}{{\mathbf 1}}
\newcommand{\ep}{\epsilon}
\newcommand{\en}{r}
\newcommand{\tu}{s}
\newcommand{\integ}{{int}}
\newcommand{\module}{{mod}}
\newcommand{\inc}{{deg}}
\newcommand{\dec}{{root}}

\newcommand{\ome}{\omega_E}

\newcommand{\bevis}{{\bf Proof. }}
\newcommand{\demofin}{\qed \vskip 3.5mm}
\newcommand{\nyp}[1]{\noindent {\bf (#1)}}
\newcommand{\demo}{{\it Proof. }}
\newcommand{\demodone}{\demofin}
\newcommand{\parg}{{\vskip 2mm \addtocounter{theorem}{1}  
                   \noindent {\bf \thetheorem .} \hskip 1.5mm }}

\newcommand{\lcm}{{\text{lcm}}}


\newcommand{\dl}{\Delta}
\newcommand{\cdel}{{C\Delta}}
\newcommand{\cdelp}{{C\Delta^{\prime}}}
\newcommand{\dlst}{\Delta^*}
\newcommand{\Sdl}{{\mathcal S}_{\dl}}
\newcommand{\lk}{\text{lk}}
\newcommand{\lkd}{\lk_\Delta}
\newcommand{\lkp}[2]{\lk_{#1} {#2}}
\newcommand{\del}{\Delta}
\newcommand{\delr}{\Delta_{-R}}
\newcommand{\dd}{{\dim \del}}

\renewcommand{\aa}{{\bf a}}
\newcommand{\bb}{{\bf b}}
\newcommand{\cc}{{\bf c}}
\newcommand{\xx}{{\bf x}}
\newcommand{\yy}{{\bf y}}
\newcommand{\zz}{{\bf z}}
\newcommand{\mv}{{\xx^{\aa_v}}}
\newcommand{\mF}{{\xx^{\aa_F}}}

\newcommand{\proj}[1]{{\mathbb P}^{#1}}
\newcommand{\hele}{{\mathbb Z}}
\newcommand{\nat}{{\mathbb N}}
\newcommand{\rat}{{\mathbb Q}}

\newcommand{\pnm}{{\bf P}^{n-1}}
\newcommand{\opnm}{{\go_{\pnm}}}
\newcommand{\op}[1]{\gO_{\proj{#1}}}
\newcommand{\ompn}{\Omega_{\proj{n}}}
\newcommand{\opn}{\op{n}}
\newcommand{\opm}{\op{m}}

\newcommand{\ompnm}{\omega_{\pnm}}

\newcommand{\dt}{{\displaystyle \cdot}}
\newcommand{\st}{\hskip 0.5mm {}^{\rule{0.4pt}{1.5mm}}}              
\newcommand{\disk}{\scriptscriptstyle{\bullet}}

\newcommand{\cF}{F_\dt}
\newcommand{\Fd}{F_{\disk}}
\newcommand{\pol}{f}

\newcommand{\disc}{\circle*{5}}

\newcommand{\Dab}{{\mathbb D}(\bfa, \bfb)}
\newcommand{\Ddab}{B(\bfa, \bfb)}
\newcommand{\Tab}{{\mathbb T}(\bfa, \bfb)}
\newcommand{\Cab}{C(\bfa, \bfb)}
\newcommand{\Pab}{B(\bfa, \bfb)}
\newcommand{\Lhkab}{L^{HK}(\bfa, \bfb)}
\newcommand{\Lab}{L(\bfa, \bfb)}
\newcommand{\Sigab}{\Sigma(\bfa, \bfb)}
\newcommand{\hlow}{h_{low}}
\newcommand{\hup}{h_{up}}
\newcommand{\up}[2]{\hup}
\newcommand{\facet}[2]{{\bf facet}(#1, #2)}
\newcommand{\BS}{Boij-S\"oderberg}
\newcommand{\bstar}{{\mathbb B}^*}
\newcommand{\Symm}{\mbox{Symm}}
\newcommand{\zinc}[1]{{\mathbb Z}^{#1}_\inc}
\newcommand{\zdec}[1]{{\mathbb Z}^{#1}_\dec}

\def\CC{{\mathbb C}}
\def\GG{{\mathbb G}}
\def\ZZ{{\mathbb Z}}
\def\NN{{\mathbb N}}
\def\RR{{\mathbb R}}
\def\OO{{\mathbb O}}
\def\QQ{{\mathbb Q}}
\def\VV{{\mathbb V}}
\def\PP{{\mathbb P}}
\def\EE{{\mathbb E}}
\def\FF{{\mathbb F}}
\def\AA{{\mathbb A}}

\section*{Introduction}

In November 2006 M.Boij and J.S\"oderberg put out on the arXiv a preprint 
"Graded
Betti numbers of Cohen-Macaulay modules and the multiplicity conjecture".
The paper concerned resolutions of graded modules over the polynomial
ring $S = \kr[x_1, \ldots, x_n]$ over a field $\kr$. It put forth
two striking conjectures on the form of their resolutions. These
conjectures and their subsequent proofs 
have put the greatest floodlight on our understanding of 
resolutions over polynomial rings since the inception of the field 
in 1890. In this year
David Hilbert published his syzygy theorem stating that a graded ideal over the
polynomial ring in $n$ variables has a resolution of length less than
or equal to $n$.  Resolutions of modules both over the
polynomial ring and other rings have since then been one of the pivotal topics
of algebraic geometry and commutative algebra, and more generally in the field
of associative algebras.

   For the next half a year after Boij and S\"oderberg put out their 
conjectures, they were
incubating
in the mathematical community, and probably not so much exposed to 
attacks. The turning point was the conference at MSRI, Berkeley in
April 2007 in honor of David Eisenbud 60'th birthday, where the conjectures
became a topic of discussion. 

For those familiar with resolutions of graded modules over the 
polynomial ring, a complete classification of their numerical invariants,
the graded Betti numbers $(\beta_{ij})$, seemed a momentous task, 
completely out of reach (and still does). Perhaps the central idea 
of Boij and S\"oderberg is this: We don't try to determine if 
$(\beta_{ij})$ are the graded Betti numbers of a module, but let us see
if we can determine if $m \cdot (\beta_{ij})$ are the graded Betti numbers 
of a module if $m$ is some big integer. 

This is the idea of {\it stability} which has been so successful in 
stable homotopy theory in algebraic topology and rational divisor theory
in algebraic geometry.

Another way to phrase the idea of Boij and S\"oderberg is that we
do not determine the graded Betti numbers $(\beta_{ij})$ but 
rather the positive rays $t \cdot (\beta_{ij})$ where $t$ is a
positive rational number. It is easy to see that these rays form
a cone in a suitable vector space over the rational numbers.

The conjectures of Boij and S\"oderberg considered the cone $B$
of such diagrams coming from modules of codimension $c$ 
with the shortest possible length of resolution, $c$. 
This is the class of Cohen-Macaulay modules.
The first conjecture states precisely what
the extremal rays of the cone $B$ are. The diagrams on these rays
are called {\it pure diagrams}. They are the possible Betti diagrams
of {\it pure resolutions}, 
\begin{equation} \label{IntroLigPure} 
S(-d_0)^{\beta_{0,d_0}} \vpil S(-d_1)^{\beta_{1,d_1}} \vpil \cdots
\vpil S(-d_c)^{\beta_{c,d_c}} 
\end{equation}
of graded modules, where the length $c$ is equal to the codimension of the
module.
To prove this conjecture involved
two tasks. The first is to show that there are vectors on these rays 
which actually are Betti diagrams of modules. The second is to show that these
rays account for {\it all} the extremal rays in the cone $B$, 
in the sense that any Betti diagram is
a positive rational combination of vectors on these rays.
This last part was perhaps what people found most suspect. 
Eisenbud has said that his immediate reaction was that this could
not be true.

  Boij and S\"oderberg made a second conjecture giving a refined
description of the cone $B$. There is a partial order on the pure
diagrams, and in any chain in this partial order the pure diagrams
are linearly independent. Pure diagrams in a chain therefore 
generate a simplicial cone. Varying over the different chains we then
get a simplicial fan of Betti diagrams. The refinement of the conjectures
states that the realization of this  simplicial fan is the positive
cone $B$. In this way each Betti diagram lies on a unique minimal face
of the simplicial fan, and so we get a strong uniqueness statement
on how to write the Betti diagram of a module.

   After the MSRI conference in April 2007, Eisenbud and
the author independently started to look into the existence question,
to construct {\it pure resolutions} whose Betti diagram is a pure 
diagram. 
They came up with the construction of the $GL(n)$-equivariant resolution
described in Subsection \ref{SubsekExiEkvi}. Jerzy Weyman was 
instrumental in proving the exactness of this resolution and
the construction was published in a joint paper in 
September 2007 on the arXiv, \cite{EFW}.
In the same paper also appeared another construction of pure 
resolutions described in Subsection \ref{SubsekExiDeg}.

   After this success D. Eisenbud and F.-O. Schreyer went on to 
work on the other part of the conjectures.
And in December 2007
they published on the arXiv the proof of
the second part of the conjectures of Boij and S\"oderberg, \cite{ES}. 
But at least two more interesting things appeared in this paper.
They gave a construction of pure resolutions that worked in all 
characteristics. The constructions above, \cite{EFW} work only 
in characteristic
zero. But most startling, they discovered a surprising duality with
cohomology tables of algebraic vector bundles on projective spaces. And fairly 
parallel to the proof of the second Boij-S\"oderberg
they were able to give a complete description of all cohomology
tables of algebraic vector bundles on projective spaces, up to 
multiplication by a positive rational number. 

    In the wake of this a range of papers have followed, most
of which are discussed in this survey. But one thing still needs to be
addressed. What enticed Boij and S\"oderberg to come up with their
conjectures? The origin here lies in an observation by Huneke and 
Miller
from 1985, that if a Cohen-Macaulay quotient
ring of $A = S/I$ has pure resolution (\ref{IntroLigPure}) 
(so $d_0 = 0$ and $\beta_{0,d_0} = 1$),
then its multiplicity $e(A)$ is equal to the surprisingly simple expression  
\[ \frac{1}{c!} \cdot \prod_{i = 1}^c d_i. \]
This led naturally to consider resolutions $\Fd$ of Cohen-Macaulay 
quotient rings $A = S/I$ in general. 
In this case one has in each homological term $F_i$ in the
resolution a maximal twist $S(-a_i)$ (so $a_i$ is minimal) and a minimal 
twist $S(-b_i)$ occurring.
The {\it multiplicity conjecture} of Herzog, Huneke and Srinivasan, see 
\cite{HS} and \cite{HM}, stated that the multiplicity of the quotient
ring $A$ is in the following range
\[ \frac{1}{c!} \prod_{i = 1}^c a_i \leq e(A) \leq 
   \frac{1}{c!} \prod_{i = 1}^c b_i. \]
Over the next two decades a substantial number of papers were published on this
treating various classes of rings, and also various generalizations of this
conjecture. But efforts in general did not succeed because of
the lack of a strong enough understanding of the (numerical) 
structure of resolutions. Boij and S\"oderberg's central idea is to see the
above inequalities as a projection of convexity properties of the
Betti diagrams of graded Cohen-Macaulay modules: The pure
diagrams generate the extremal rays in the cone of Betti diagrams.

\medskip
\noindent {\it Notation.} 
The graded Betti numbers $\beta_{ij}(M)$ of a finitely generated
module $M$ are indexed by 
$i = 0, \ldots, n$ and $j \in \hele$. Only a finite number of 
these are nonzero. By a {\it diagram} we shall mean a collection 
of rational numbers $(\beta_{ij})$, indexed as above, with only a finite
number of them being nonzero.

\medskip
The organization of this paper is as follows. In Section \ref{SekRes} we
give the important notions, like the graded Betti numbers of 
a module, pure resolutions and Cohen-Macaulay modules.
Such modules have certain linear constraints on their
graded Betti numbers, the Herzog-K\"uhl equations, giving
a subspace $L^{HK}$ of the space of diagrams. We define the 
positive cone $B$ in $L^{HK}$ of Betti diagrams of Cohen-Macaulay modules. 
An important technical convenience is that we fix a ``window'' on
the diagrams, considering Betti diagrams where the $\beta_{ij}$
are nonzero only in a finite range of indices $(i,j)$. This makes the Betti
diagrams live in a finite dimensional vector space.
Then we present
the Boij-S\"oderberg conjectures. We give both the algorithmic version,
concerning the decomposition of Betti diagrams,
and the geometric version in terms of fans.

In Section \ref{SekFacet} we define the simplicial fan $\Sigma$ of diagrams. 
The goal is to show that its realization is the positive cone $B$, and
to do this we study the exterior facets of $\Sigma$. 
The main work of this section is to find the equations of these facets. 
They are the key to the duality with algebraic vector bundles, 
and the form of their equations are derived from suitable 
pairings between Betti diagrams and cohomology tables of
vector bundles.
The positivity of the pairings proves that the cone $B$ is contained
in the realization of $\Sigma$, which is one part of the conjectures.

The other part, that $\Sigma$ is contained in $B$, is shown in Section 3
by providing the existence of pure resolutions. We give in 
\ref{SubsekExiEkvi} the construction of the equivariant pure resolution
of \cite{EFW}, in \ref{SubsekExiChfree} the characteristic free resolution
of \cite{ES}, and in \ref{SubsekExiDeg} the second construction of 
\cite{EFW}. For cohomology of vector bundles, the bundles with 
supernatural cohomology play the analog role of pure resolutions. 
In \ref{SubsekExiEkvisup} we give the equivariant construction of 
supernatural bundles, and in \ref{SubsekExiChfreesup}
the characteristic free construction of \cite{ES}. 

In Section \ref{SekBunt} we first consider the cohomology of vector
bundles on projective spaces, and give the complete classification of
such tables up to multiplication by a positive rational number.
The argument runs analogous to what we do for Betti diagrams. 
We define the positive cone of cohomology tables $C$, and the simplicial
fan of tables $\Gamma$. We compute the equations of the exterior facets
of $\Gamma$ which again are derived from the pairings between 
Betti diagrams and cohomology tables.  The positivity of these
pairings show that $C \sus \Gamma$, and the existence of supernatural 
bundles that $\Gamma \sus C$, showing the desired equality $C = \Gamma$. 

Section \ref{SekExt} considers extensions of the previous results.
First in \ref{SubsekFurNC} we get the  
classification of graded Betti numbers of {\it all} 
modules up to positive rational multiples.
For cohomology of coherent sheaves there is not yet a classification, but
in \ref{SubsekBuntKnipp} the procedure to decompose cohomology tables
of vector bundles is extended to cohomology tables of coherent
sheaves. However this procedure involves an infinite number of steps,
so this decomposition involves an infinite sum.

Section \ref{SekFur} gives more results that have followed in the wake 
of the conjectures and their proofs.
The ultimate goal, to classify Betti diagrams of modules (not
just up to rational multiple) is considered in \ref{SubsekFurInt},
and consists mainly of examples of diagrams which are or are not the
Betti diagrams of modules. So far we have considered $\kr[x_1, \ldots, x_n]$
to be standard graded, i.e. each $\deg x_i = 1$. In \ref{SubsekFurGrad}
we consider other gradings and multigradings on the $x_i$. 
Subsection \ref{SubsekFurPoset} considers the partial order on pure
diagrams, so essential in defining the simplicial fan $\Sigma$. 
In \ref{SubsekFurComp} we inform on 
computer packages related to Boij-S\"oderberg theory, and in 
\ref{SubsekFurPro} we give some important open problems.

\medskip
\noindent {\it Acknowledgement.}{ We thank the referee for several corrections
and useful suggestions for improving the presentation.}

\section{The Boij-S\"oderberg conjectures}
\label{SekRes}

We work over the standard graded polynomial ring $S = \kr[x_1, \ldots, x_n]$.
For a graded module $M$ over $S$, we denote by $M_d$ its graded piece
of degree $d$, and by $M(-r)$ the module where degrees are shifted
so that $M(-r)_d = M_{d-r}$. 

\medskip
\noindent{\it Note.}
We shall always assume our modules to be
finitely generated and graded. 

\subsection{Resolutions and Betti diagrams}
A natural approach to understand such 
modules is to understand their numerical invariants. The most
immediate of these is of course the Hilbert function:
\[  h_M(d) = \dim_{\kr} M_d. \]
Another set of invariants is obtained by considering its minimal graded
free resolution:
\begin{equation}
\label{BSTLigRes} F_0 \vpil F_1 \vpil F_2 \vpil \cdots \vpil F_l
\end{equation}
Here each $F_i$ is a graded free $S$-module 
$\oplus_{j \in \hele} S(-j)^{\beta_{ij}}$. 

\begin{example} \label{BSTEksM}
Let $S = \kr[x,y]$ and $M$ be the quotient ring $S/(x^2, xy, y^3)$.
Then its minimal resolution is 
\[  S \xleftarrow{ \left [ \begin{matrix} x^2 &  xy &  y^3 
\end{matrix} \right ]}
S(-2)^2 \oplus S(-3)
\xleftarrow{\left [ \begin{matrix} y & 0 \\ -x & y^2 \\ 0 & -x \end{matrix} 
\right ]}
S(-3) \oplus S(-4). \]
\end{example}

The multiple $\beta_{i,j}$ of the term $S(-j)$ in the $i$'th 
homological part $F_i$ of the resolution, is called the $i$'th graded 
Betti number of degree $j$. These Betti numbers constitute another 
natural set of numerical invariants, and the ones that are the topic
of the present notes. By the resolution (\ref{BSTLigRes}) we see that
the graded Betti number determine the Hilbert function of $M$. In fact
the dimension $\dim_{\kr} M_d$ is the alternating sum 
$\sum (-1)^i \dim_\kr (F_i)_d$.  
The Betti numbers are however more refined numerical
invariants of graded modules than the Hilbert function.

\begin{example}
Let $M^\prime$ be the quotient ring $S/(x^2, y^2)$. Its minimal 
free resolution is 
\[ S \vpil S(-2)^2 \vpil S(-4) . \]
Then $M$ of Example \ref{BSTEksM} and $M^\prime$ have the same
Hilbert functions, but their graded Betti numbers are different.
\end{example}

\medskip

The Betti numbers are usually displayed in an array. The immediate natural
choice is to put $\beta_{i,j}$ in the $i$'th column and $j$'th row, 
so the diagram of Example \ref{BSTEksM} would be:


\begin{equation*}
\begin{matrix} 0 \\ 1 \\ 2 \\ 3 \\ 4
\end{matrix} 
\overset{ \begin{matrix} 0 & 1 & 2 \end{matrix} }
{\left [ 
\begin{matrix}
1 & 0 & 0 \\
0 & 0 & 0 \\
0 & 2 & 0 \\
0 & 1 & 1 \\
0 & 0 & 1
\end{matrix}
\right ]}
\end{equation*}

However, to reduce the number of rows, one uses the convention that
the $i$'th column is shifted $i$ steps up. Thus $\beta_{i,j}$ is put in the
$i$'th column and the $j-i$'th row. Alternatively, $\beta_{i,i+j}$ is 
put in the $i$'th column and $j$'th row. So the diagram above 
is displayed as :

\begin{equation}
\label{ResLigBedia}
\begin{matrix} 0 \\ 1 \\ 2
\end{matrix}
\overset{\begin{matrix} 0 & 1 & 2 \end{matrix}}
{\left [
\begin{matrix}
1 & 0 & 0 \\
0 & 2 & 1 \\
0 & 1 & 1
\end{matrix}
\right ]}
\end{equation}

A Betti diagram has columns indexed by $0, \ldots, n$ and rows indexed
by elements of $\hele$, 
but any Betti diagram (of a finitely generated graded module)
is nonzero only in a finite number of rows.
Our goal is to understand the possible Betti diagrams that can occur
for Cohen-Macaulay modules. This objective
seems however as of yet out of reach. The central idea of Boij-S\"oderberg
theory is rather to describe Betti diagrams up to a multiple by a rational
number. I.e. we do not determine if a diagram $\beta$ is
a Betti diagram of a module, but we will determine if $q\beta$ is
a Betti diagram for some positive rational number $q$.
By Hilbert's Syzygy Theorem we know that the length $l$ of the resolution
(\ref{BSTLigRes}) is $\leq n$. 
Thus we consider Betti diagrams to live in the $\rat$-vector space
${\mathbb D} = \oplus_{j \in \hele} {\rat}^{n+1}$,
with the $\beta_{ij}$ as coordinate functions. 
An element in this vector space, a collection of rational numbers
$(\beta_{ij})_{{i = 0, \ldots, n},{j \in \hele}}$ where only a finite number
is nonzero, is called a {\it diagram}. 

\subsection{The positive cone of Betti diagrams}

We want to make our Betti diagram live in a finite dimensional vector
space, so we fix a ``window'' in ${\mathbb D}$ as follows.
 Let $c \leq n$ and $\zinc{c+1}$ be the set of strictly increasing 
integer sequences
$(a_0, \ldots, a_c)$ in $\hele^{c+1}$. Such an element is
called a {\it degree sequence}. Then $\zinc{c+1}$ is a partially
ordered set with $\bfa \leq \bfb$ if $a_i \leq b_i$ for all 
$i = 0, \ldots, c$.
\begin{definition}
 For $\bfa, \bfb$ in $\zinc{c+1}$ let $\Dab$
be the set of diagrams $(\beta_{ij})_{{i = 0, \ldots, n},{j \in \hele}}$
such that $\beta_{ij}$ may be nonzero only in the range 
$0 \leq i \leq c$ and
$a_i \leq j \leq b_i$.
\end{definition}
 
We see that $\Dab$ is simply the $\rat$-vector space with basis
elements indexed by the pairs $(i,j)$ in the range above determined
by $\bfa$ and $\bfb$.
The 
diagram of Example \ref{BSTEksM}, displayed above in (\ref{ResLigBedia}), 
lives in the window $\Dab$ with  
$\bfa = (0,1,2)$ and $\bfb = (0,3,4)$ (or 
$\bfa$ any triple $\leq (0,1,2)$ and 
$\bfb$ any triple $\geq (0,3,4)$).

 If the module $M$ has codimension $c$, equivalently its Krull dimension 
is $n-c$, the depth of $M$ is $\leq n-c$. By the Auslander-Buchsbaum
theorem, \cite{Ei}, the length of the resolution is $l \geq c$.
To make things simple we assume that $l$ has its smallest possible 
value $l = c$ or equivalently that $M$ has depth equal to the dimension
$n-c$. This gives the class of Cohen-Macaulay (CM) modules. 
\begin{definition}
Let $\bfa$ and $\bfb$ be in $\zinc{c+1}$.
\begin{itemize}
\item $L(\bfa, \bfb)$ is the $\rat$-vector
subspace of the window $\Dab$ spanned by the Betti
diagrams of CM-modules of codimension $c$, whose 
Betti diagrams are in this window.
\item $B(\bfa,\bfb)$ is the set of non-negative rays spanned
by such Betti diagrams.
\end{itemize}
\end{definition}

\begin{lemma} $B(\bfa, \bfb)$ is a cone.
\end{lemma}

\begin{proof}
We must show that if $\beta_1$ and $\beta_2$ are in $B(\bfa, \bfb)$ then
$q_1 \beta_1 + q_2 \beta_2$ is in $B(\bfa, \bfb)$ for all positive
rational numbers $q_1$ and $q_2$. 

This is easily seen to be equivalent to the following: Let $M_1$ and
$M_2$ be CM-modules of codimension $c$ with Betti diagrams in 
$\Dab$. Show that $c_1 \beta(M_1) + c_2\beta(M_2)$ is 
in $B(\bfa, \bfb)$ for all natural numbers $c_1 $ and $c_2$. But this
linear combination is clearly the Betti diagram of the 
CM-module
$M_1^{c_1} \oplus M_2^{c_2}$ of codimension $c$. 
And clearly this linear combination is still in the window $\Dab$.
\end{proof}

Our main objective is to describe this cone.

\subsection{Herzog-K\"uhl equations}

Now given a resolution (\ref{BSTLigRes}) of a module $M$, 
there are natural relations its Betti numbers $\beta_{ij}$ must
fulfil. First of all if the codimension $c \geq 1$, then clearly the 
alternating sum of the ranks of the $F_i$ must be zero. I.e.
\[ \sum_{i,j} (-1)^i \beta_{ij} = 0. \]
When the codimension $c \geq 2$ we get more numerical restrictions. 
Since $M$ has dimension $n-c$, its Hilbert series is of the form
$h_M(t) = \frac{p(t)}{(1-t)^{n-c}}$, where $p(t)$ is some polynomial.
This may be computed as the alternating sum of the Hilbert series of
each of the terms in the resolution (\ref{BSTLigRes}):
\[ h_M(t) = \frac{\sum_j \beta_{0j} t^j}{(1-t)^n}
- \frac{\sum_j \beta_{1j} t^j}{(1-t)^n}
+ \cdots + (-1)^l  \frac{\sum_j \beta_{lj} t^j}{(1-t)^n}
. \]
Multiplying with $(1-t)^n$ we get
\[ (1-t)^c p(t) =  \sum_{i,j}  (-1)^i\beta_{ij}t^j. \]
Differentiating this successively and setting $t = 1$, gives the
equations
\begin{equation} \label{BSTLigHK} \sum_{i,j} (-1)^i j^p \beta_{ij} = 0, 
\quad p = 0, \ldots, c-1. 
\end{equation}
These equations are the {\it Herzog-K\"uhl equations} for the Betti diagram
$(\beta_{ij})$ of a module of codimension $c$.

We denote by $L^{HK}(\bfa, \bfb)$ the
$\rat$-linear subspace of diagrams in $\Dab$ fulfilling the Herzog-K\"uhl 
equations  (\ref{BSTLigHK}). 
Note that $L(\bfa, \bfb)$ is a subspace of $L^{HK}(\bfa, \bfb)$.
We shall show that these spaces are equal.


\subsection{Pure resolutions} \label{ResSubsecPure}
Now we shall consider a particular case of the resolution (\ref{BSTLigRes}).
Let $\bfd = (d_0, \ldots, d_l)$ be a strictly increasing sequence of
integers, a degree sequence. 
The resolution (\ref{BSTLigRes}) is {\it pure} if it
has the form 
\begin{equation*} 
S(-d_0)^{\beta_{0,d_0}} \vpil S(-d_1)^{\beta_{1,d_1}} \vpil \cdots
\vpil S(-d_l)^{\beta_{l,d_l}}. 
\end{equation*}

By a {\it pure diagram} (of type $\bfd$) we shall mean a diagram such that 
for each column $i$ there is only one nonzero entry $\beta_{i, d_i}$, and
the $d_i$ form an increasing sequence. We see that a pure resolution
gives a pure Betti diagram. 

When $M$ is CM of codimension $c$, the Herzog-K\"uhl equations give 
the following set of equations
\[ \left[ \begin{matrix} 1 & - 1 & \cdots & (-1)^{c} \\
               d_0 & -d_1 & \cdots & (-1)^{c}d_{c} \\
               \vdots & & & \vdots \\
               d_0^{c-1} & - d_1^{c-1} & \cdots & (-1)^{c}d_c^{c-1}
          \end{matrix} \right ] 
\left[ \begin{matrix} \beta_{0, d_0} \\ \beta_{1, d_1} \\ \vdots
                     \\ \beta_{c, d_c}  \end{matrix}
\right ] . 
\]

This is a $c \times (c+1)$ matrix of maximal rank. Hence there is only 
a one-dimensional $\rat$-vector space of solutions. The solutions may
be found by computing the maximal minors which are Vandermonde determinants
and we find
\[ \beta_{i, d_i} =  (-1)^i \cdot t \cdot \prod_{k \neq i} \frac{1}{(d_k - d_i)} \]
where $t \in \rat$. When $t > 0$ all these are positive. 
Let $\pi(\bfd)$ be the diagram which is the smallest integer solution
to the equations above.
As we shall see pure resolutions and pure diagrams play a central role
in the description of Betti diagrams up to rational multiple. 

\subsection{Linear combinations of pure diagrams}
The rays generated by the $\pi(\bfd)$ turn out to be exactly the 
extremal rays in the cone $B(\bfa, \bfb)$. 
Thus any Betti diagram is a positive linear combination of pure diagrams.
Let us see how this works
in an example.

\begin{example} \label{BSTEksDek} If the diagram of Example \ref{BSTEksM}
\[ \beta = \begin{matrix} 0 \\ 1 \\ 2 \end{matrix}
\left [ \begin{matrix} 1 & 0 & 0 \\ 
                          0 & 2 & 1 \\
                          0 & 1 & 1 
           \end{matrix}  \right ] \]
is a positive linear combination of pure diagrams $\pi(\bfd)$,
the only possibilities for these diagrams are
\[ \pi(0,2,3) = \left [ \begin{matrix} 1 & 0 & 0 \\
                0 & 3 & 2 \\
                0 & 0 & 0
                \end{matrix} \right ], \quad
  \pi(0,2,4) = \left [ \begin{matrix} 1 & 0 & 0 \\
                0 & 2 & 0 \\
                0 & 0 & 1
                \end{matrix} \right ], \quad
\pi(0,3,4) = \left [ \begin{matrix} 1 & 0 & 0 \\
                0 & 0 & 0 \\
                0 & 4 & 3
                \end{matrix} \right ]. \]
Note that by the natural partial order on degree sequences we have
\[ (0,2,3) < (0,2,4) < (0,3,4). \]
To find this linear combination we proceed as follows. Take the largest
positive multiple $c_1$ of $\pi(0,2,3)$ such that
$\beta - c_1 \pi(0,2,3)$ is still non-negative. We see that $c_1 = 1/2$ and
get 
\[ \beta_1 = \beta - \frac{1}{2} \pi(0,2,3) = \left[
   \begin{matrix} 1/2 & 0 & 0 \\
          0 & 1/2 & 0 \\
          0 & 1 & 1 
          \end{matrix} \right ] .  \]
Then take the largest possible multiple $c_2$ of $\pi(0,2,4)$ such that
$\beta_1 - c_2 \pi (0,2,4)$ is non-negative. We see that $c_2 = 1/4$ and
get 
\[ \beta_2 = \beta - \frac{1}{2} \pi(0,2,3) - \frac{1}{4} \pi(0,2,4)
= \left [ \begin{matrix} 1/4 & 0 & 0 \\
          0 & 0 & 0 \\
          0 & 1 & 3/4 
         \end{matrix}  \right ] . \]
Taking the largest multiple $c_3$ of $\pi(0,3,4)$ such that
$\beta_2 - c_3 \pi (0,3,4)$ is non-negative, we see that $c_3 = 1/4$ and 
the last expression becomes the zero diagram. Thus we get $\beta$ as a positive
rational combination of pure diagrams
\[ \beta = \frac{1}{2} \pi(0,2,3) + \frac{1}{4} \pi(0,2,4) + 
\frac{1}{4} \pi(0,3,4). \]
\end{example}

The basic part of Boij-S\"oderberg theory says that this procedure
will always work: It gives a {\it non-negative} linear combination of
pure diagrams.  We proceed to develop this in more detail.
With $\hele^{c+1}_\inc$ equipped with the natural partial order, we get
for $\bfa, \bfb \in \hele^{c+1}_\inc$ the interval $[\bfa, \bfb]_\inc$
consisting of all degree sequences $\bfd$ with $\bfa \leq \bfd \leq \bfb$. 
The diagrams $\pi(\bfd)$ where $\bfd \in [\bfa, \bfb]_\inc$ are the pure diagrams
in the window determined by $\bfa$ and $\bfb$.

\begin{example}
If $\bfa = (0,2,3)$ and $\bfb = (0,3,4)$, the vector space 
$\Dab$ consists of the diagrams which may be nonzero
in the positions marked by $*$ below.

\[ \left [ \begin{matrix} * & 0 & 0 \\
         0 & * & * \\
         0 & * & * 
   \end{matrix} \right ],  \]
and so is five-dimensional.
The Herzog-K\"uhl equations for the diagrams ($c = 2$) are the following two
equations
\begin{eqnarray*}
\beta_{0,0} - (\beta_{1,2} + \beta_{1,3}) + (\beta_{2,3} + \beta_{2,4}) 
&  = & 0 \\
0 \cdot \beta_{0,0} - (2\beta_{1,2} + 3\beta_{1,3}) + 
(3\beta_{2,3} + 4\beta_{2,4})  &  = & 0 
\end{eqnarray*}
These are linearly independent and so $L^{HK}(\bfa, \bfb)$ will be 
three-dimensional. On the other hand the diagrams $\pi(0,2,3), \pi(0,2,4)$
and $\pi(0,3,4)$ are clearly linearly independent in this vector
space and so they form a basis for it. This is a general phenomenon.
\end{example}

The linear space $L^{HK}(\bfa, \bfb)$ (and as will turn out $L(\bfa, \bfb)$) 
may be described as follows. 

\begin{proposition} \label{BSTProLin} Given any maximal chain 
\[ \bfa = \bfd^1 < \bfd^2 < \cdots < \bfd^r = \bfb \]
in $[\bfa,\bfb]_\inc$. The associated pure diagrams
\[ \pi(\bfd^1), \pi(\bfd^2), \ldots, \pi (\bfd^r) \]
form a basis for $L^{HK}(\bfa, \bfb)$.  The length of such
a chain, and hence the dimension of the latter vector space is
$r = 1 + \sum (b_i - a_i)$. 
\end{proposition}

\begin{proof}
Let $\beta$ be a solution of the HK-equations contained in the
window $\Dab$. The vectors $\bfd^1$ and $\bfd^2$ differ in
one coordinate, suppose it is the $i$'th coordinate, 
so $\bfd^1 = (\cdots, d^1_{i}, \cdots)$ and 
$\bfd^2 = (\cdots, d^1_{i} + 1, \cdots)$.
 Let $c_1$ be such that
$\beta_1 = \beta - c_1 \pi(\bfd^1)$ is zero in position $(i,d^1_{i})$. 
Then $\beta_1$ is contained in the window ${\mathbb D}(\bfd^2, \bfb)$ and
$\bfd^2, \ldots, \bfd^r$ is a maximal chain in $[\bfd^2, \bfb]_\inc$.
We may proceed by induction and in the end get
$\beta_{r-1}$ contained in $[\bfb, \bfb]_\inc$. Then $\beta_{r-1}$ is pure
and so is a multiple of $\pi(\bfd^r)$. In conclusion 
\[ \beta = \sum_{i = 1}^r c_i \pi(\bfd^i). \]
To see that the $\pi(\bfd^i)$ are linearly independent, note that 
$\pi(\bfd^1)$ is not a linear combination of the $\pi(\bfd^i)$ for 
$i \geq 2$ since $\pi(\bfd^1)$ is nonzero in position $(i,d^1_{i})$
while the $\pi(\bfd^i)$ for $i \geq 2$ are zero in this position.
Hence a dependency must involve only $\pi(\bfd^i)$ for $i \geq 2$.
But then we may proceed by induction.
\end{proof}

\subsection{The Boij-S\"oderberg conjectures}
The first part of the original Boij-S\"oderberg conjectures states
the following.

\begin{theorem} \label{ResTheBS1} 
For every degree sequence $\bfd$, a strictly increasing sequence of integers
$(d_0 , \ldots, d_c)$, there exists a Cohen-Macaulay module
$M$ of codimension $c$ with pure resolution of type $\bfd$. 
\end{theorem}

We shall in Section \ref{ExiSek} give an overview of the constructions of such
resolutions, making the conjecture a theorem.

\begin{corollary} The linear space $L(\bfa, \bfb)$ is equal to
$L^{HK}(\bfa, \bfb)$. 
\end{corollary}

\begin{proof} The diagram $\pi(\bfd)$ may now be realized, up to multiplication
by a scalar, as the Betti diagram of a Cohen-Macaulay module.
\end{proof}

The second part of the Boij-S\"oderberg conjectures says the following.

\begin{theorem} \label{ResTheBS2} Let $M$ be a Cohen-Macaulay module of 
codimension $c$
with Betti diagram $\beta(M)$ in $\Dab$. There is
a unique chain 
\[ \bfd^1 < \bfd^2 < \cdots < \bfd^r \]
in $[\bfa, \bfb]_\inc$ such that $\beta(M)$ is uniquely a linear combination
\[ c_1 \pi(\bfd^1) + c_2 \pi(\bfd^2) + \cdots + c_r \pi(\bfd^r) \]
where the $c_i$ are positive rational numbers.
\end{theorem}

\begin{remark}
When $M$ is any graded module of codimension $\geq c$, 
the same essentially holds true, but one must allow degree sequences
$\bfd^i$ in $\hele_{\deg}^p$ where $p$ ranges over $c+1, \ldots, n+1$. 
See Subsection \ref{SubsekFurNC}.
\end{remark}

\begin{remark}
Combining this with Theorem \ref{ResTheBS1}  
we see that there are modules $M_i$
with pure resolution of type $\bfd^i$ such that for suitable multiples $p$
and $p_i$ 
then $M^p$ and $\oplus_i M_i^{p_i}$ have the same Betti diagram.
\end{remark}

\subsection{Algorithmic interpretation} \label{ResSubsecAlg}
As a consequence of Theorem \ref{ResTheBS2} we get a simple
algorithm to find this unique decomposition, which is the way we
did it in Example \ref{BSTEksDek}. This algorithm, with interesting
consequences, is presented in \cite[Section 1]{ES}. 
For a diagram $\beta$, for each $i$
let $d_i$ be the minimal $j$ such that $\beta_{ij}$ is nonzero.
This gives a sequence $\ubfd(\beta) = (d_0, d_1, \ldots, d_c)$,
the {\it lower bound} of $\beta$. 

\begin{example} Below the nonzero positions of $\beta$ is indicated by 
$*$'s. 

\begin{equation*}
\begin{matrix} 0 \\ 1 \\ 2 \\ 3 
\end{matrix}
\left [ \begin{matrix}
* & * & 0 & 0 & 0 \\
* & * & * & 0 & 0 \\
0 & * & * & * & * \\
0 & 0 & 0 & * & * 
\end{matrix} \right ]  
\end{equation*}

 Then $\ubfd(\beta) = (0,1,3,5,6)$. 
\end{example}

There is a pure Betti diagram $\pi(\ubfd(\beta))$ and let $c(\beta) > 0$ be 
the maximal
number such that $\beta^\prime = \beta - 
c(\beta) \pi(\ubfd(\beta))$ is nonnegative.

Let $M$ be a Cohen-Macaulay module. 
The algorithm is now as follows.

\begin{itemize}
\item[1.] Let $\beta = \beta(M)$ and $i = 1$.
\item[2.] Compute $\ubfd^i := \ubfd(\beta)$ and $c^i := c(\beta)$. 
Then $\ubfd^i$ will be a strictly increasing sequence.
Let $\beta := \beta - c^i \pi (\ubfd^i)$.
\item[3.] If $\beta$ is nonzero let $i := i+1$ and continue with Step 2.
Otherwise stop.
\end{itemize}

The output will then be the unique decomposition
\[\beta(M) = c^1 \pi(\ubfd^1) + c^2 \pi(\ubfd^2) + \cdots + 
c^r \pi(\ubfd^r). \]

\subsection{Geometric interpretation}
Since for any chain 
$D : \bfd^1 < \bfd^2 < \cdots < \bfd^r$ in $[\bfa, \bfb]_\inc$
the Betti diagrams $\pi(\bfd^1), \ldots, \pi(\bfd^r)$ are linearly 
independent diagrams in $\Dab$, their positive rational
linear combinations give a simplicial cone $\sigma(D)$ in $\Dab$,
which actually is in the subspace $L^{HK}(\bfa, \bfb)$. 
Two such cones will intersect along another such
cone, which is the content of the following.

\begin{proposition}
The set of simplicial cones $\sigma(D)$ 
where $D$ ranges over all chains $\bfd^1 < \cdots < \bfd^r$ in 
$[\bfa, \bfb]_\inc$ form a simplicial fan, which we denote as 
$\Sigab$. 
\end{proposition}

\begin{proof}
Let $D$ be a chain like above and $E$ another chain $\bfe^1 < \cdots < \bfe^s$
in $[\bfa, \bfb]_\inc$. We shall show that $\sigma(D)$ and $\sigma(E)$ intersect
in $\sigma (D \cap E)$. So consider 
\[ \beta = \sum c_i \pi(\bfd^i) = \sum c_i^\prime \pi(\bfe^i) \]
in the intersection. By omitting elements in the chain we may assume
all $c_i$ and $c_i^\prime$ positive. Then the lower bound of $\beta$ which
we denoted $\ubfd(\beta)$, will be $\bfd^1$. But it will also be 
$\bfe^1$, and so $\bfe^1 = \bfd^1$. Assume say that $c_1 \leq c_1^\prime$.
Let $\beta^\prime = \beta- c_1 \pi(\bfd^1)$. Then $\beta^\prime$
is in $\sigma(D \backslash \{ \bfd^1 \})$ and in $\sigma(E)$. 
By induction on the sum of the cardinalities of $D$ and $E$, we
get that $\beta^\prime$ is in $\sigma (D \cap E \backslash \{\bfd^1 \})$
and so $\beta$ is in $\sigma ( D \cap E)$. 
\end{proof}

We now get the following description of the positive cone $B(\bfa, \bfb)$.

\begin{theorem} \label{ResTheGeo}

\begin{itemize}
\item[a.] The realization of the fan $\Sigma(\bfa, \bfb)$ is contained
in the positive cone $B(\bfa, \bfb)$.
\item[b.] The positive cone $B(\bfa, \bfb)$ is contained in the 
realization of the fan $\Sigma(\bfa, \bfb)$. 
\end{itemize}
In conclusion the realization of the fan $\Sigma(\bfa, \bfb)$ 
is equal to the positive cone $B(\bfa, \bfb)$. 
\end{theorem}

It may seem overly pedantic to express it in this way but the reason should
be clear from the proof.

\begin{proof}
Part a. is equivalent to the first part of the Boij-S\"oderberg conjectures, 
Theorem \ref{ResTheBS1}. Part b. is
equivalent to the second part of the Boij-S\"oderberg conjectures, Theorem
\ref{ResTheBS2}.
\end{proof}

\section{The exterior facets of the Boij-S\"oderberg fan and their
supporting hyperplanes}
\label{SekFacet}
In order to prove Theorem \ref{ResTheBS2}, which is equivalent to 
part b. of Theorem \ref{ResTheGeo}, we must describe the exterior facets of the
Boij-S\"oderberg fan $\Sigma(\bfa, \bfb)$ and their supporting hyperplanes.

\subsection{The exterior facets}
Let $D : \bfd^1 < \cdots < \bfd^r$ be a maximal chain in $[\bfa, \bfb]_\inc$.
The positive rational linear combinations of the 
pure diagrams $\pi(\bfd^1), \ldots, \pi(\bfd^r)$ is a maximal simplicial 
cone $\sigma(D)$ in the Boij-S\"oderberg fan
$\Sigma(\bfa, \bfb)$. The facets of the cone $\sigma(D)$ are the cones
$\sigma(D \backslash \{ \bfd^i \})$ for $i = 1, \ldots, r$. 
We call such a facet {\it exterior} if it is on only one simplicial cone
in the fan $\Sigma(\bfa, \bfb)$.

\begin{example} \label{ExtEksChains}
Let $\bfa = (0,1,3)$ and $\bfb = (0,3,4)$. The Hasse diagram of the
poset $[\bfa, \bfb]_\inc$ is the diagram.

\vskip 4mm
\hskip 4cm
\scalebox{1} 
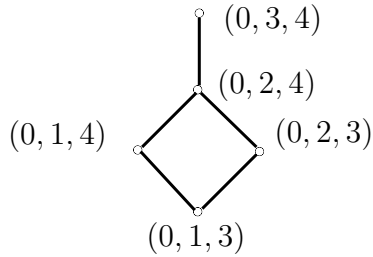
{
\begin{pspicture}(0,-1.6592188)(5.0871873,1.6592188)
\psdots[dotsize=0.12,fillstyle=solid,dotstyle=o](2.601875,1.5373437)
\psdots[dotsize=0.12,fillstyle=solid,dotstyle=o](2.581875,0.51734376)
\psdots[dotsize=0.12,fillstyle=solid,dotstyle=o](1.781875,-0.28265625)
\psdots[dotsize=0.12,fillstyle=solid,dotstyle=o](2.581875,-1.1026562)
\psdots[dotsize=0.12,fillstyle=solid,dotstyle=o](3.401875,-0.30265626)
\psline[linewidth=0.04cm](2.601875,1.4773438)(2.601875,0.57734376)
\psline[linewidth=0.04cm](1.841875,-0.24265625)(2.561875,0.47734374)
\psline[linewidth=0.04cm](2.621875,0.47734374)(3.361875,-0.24265625)
\psline[linewidth=0.04cm](2.641875,-1.0626563)(3.361875,-0.34265625)
\psline[linewidth=0.04cm](1.821875,-0.30265626)(2.541875,-1.0826563)
\usefont{T1}{ptm}{m}{n}
\rput(3.5754688,1.4473437){$(0,3,4)$}
\usefont{T1}{ptm}{m}{n}
\rput(4.255469,-0.07265625){$(0,2,3)$}
\usefont{T1}{ptm}{m}{n}
\rput(2.5554688,-1.4526563){$(0,1,3)$}
\usefont{T1}{ptm}{m}{n}
\rput(0.71546876,-0.09265625){$(0,1,4)$}
\usefont{T1}{ptm}{m}{n}
\rput(3.4954689,0.58734375){$(0,2,4)$}
\end{pspicture} 
}

There are two maximal chains in this diagram
\begin{eqnarray*}
D & : & (0,1,3) < (0,2,3) < (0,2,4) < (0,3,4) \\
E & : & (0,1,3) < (0,1,4) < (0,2,4) < (0,3,4)
\end{eqnarray*}
so the realisation of the Boij-S\"oderberg fan consists of
the union of two simplicial cones of dimension four. 
We intersect this transversally with
a hyperplane to get a three-dimension picture of this
as the union of two tetrahedra. (The vertices are labelled by the pure
diagrams on their rays.)

\hskip 3cm
\scalebox{1} 
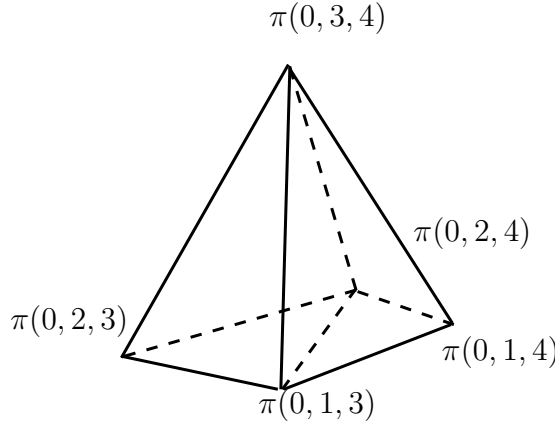
{
\begin{pspicture}(0,-2.785625)(7.5871873,2.785625)
\psline[linewidth=0.04cm](3.681875,1.96375)(1.461875,-1.93625)
\psline[linewidth=0.04cm](3.701875,1.94375)(3.581875,-2.35625)
\psline[linewidth=0.04cm](3.701875,1.90375)(5.861875,-1.47625)
\psline[linewidth=0.04cm,linestyle=dashed,dash=0.16cm 0.16cm](4.581875,-1.01625)(3.601875,-2.35625)
\psline[linewidth=0.04cm,linestyle=dashed,dash=0.16cm 0.16cm](4.601875,-1.05625)(5.881875,-1.49625)
\psline[linewidth=0.04cm,linestyle=dashed,dash=0.16cm 0.16cm](4.581875,-1.03625)(3.701875,1.92375)
\psline[linewidth=0.04cm,linestyle=dashed,dash=0.16cm 0.16cm](1.501875,-1.91625)(4.581875,-1.03625)
\psline[linewidth=0.04cm](1.461875,-1.91625)(3.541875,-2.35625)
\psline[linewidth=0.04cm](3.621875,-2.35625)(5.841875,-1.49625)
\usefont{T1}{ptm}{m}{n}
\rput(4.2054687,2.57375){$\pi(0,3,4)$}
\usefont{T1}{ptm}{m}{n}
\rput(6.505469,-1.86625){$\pi(0,1,4)$}
\usefont{T1}{ptm}{m}{n}
\rput(4.065469,-2.56625){$\pi(0,1,3)$}
\usefont{T1}{ptm}{m}{n}
\rput(0.76546876,-1.44625){$\pi(0,2,3)$}
\usefont{T1}{ptm}{m}{n}
\rput(6.1254687,-0.28625){$\pi(0,2,4)$}
\end{pspicture} 
}

There is one interior facet of the fan, while all other facets are 
exterior. The exterior facets are of three types. We give an example
of each case by giving the chain.

\begin{itemize} 
\item[1.] $D \backslash \{ (0,1,3)\}$. Here we omit the minimal 
element $\bfa$. Clearly this can only be completed to a maximal chain
in one way so this gives an exterior facet.
\item[2.] $E \backslash \{ (0,2,4) \}$. This chain contains 
$(0,1,4)$ and $(0,3,4)$. Clearly the only way to complete this to a
maximal chain is by including $(0,2,4)$, so this gives an exterior facet.
\item[3.] $D \backslash \{(0,2,4)\}$. This contains $(0,2,3)$ and 
$(0,3,4)$. When completing this to a maximal chain clearly one must first
increase the last $3$ in $(0,2,3)$ to $4$, giving $(0,2,4)$. So $D$ is the only 
maximal chain containing this.
\end{itemize}
\end{example}

The following tells that these three types are the only ways of getting
exterior facets.

\begin{proposition} \label{FacProFacet} Let $D$ be a maximal chain in 
$[\bfa, \bfb]_\inc$ and 
$f \in D$. Then $\sigma(D \backslash \{ f \})$ is an exterior facet
iff one of the following holds.
\begin{itemize}
\item[1.] $f$ is either $\bfa$ or $\bfb$.
\item[2.] The degree sequences of $f^-$ and $f^+$ immediately before
and after $f$ in $D$ differ in exactly one position.  
So for some $r$ we have 
\[ f^- = (\cdots, r-1, \cdots), f = (\cdots, r, \cdots), 
f^+ = (\cdots, r+1, \cdots). \]
\item[3.] The degree sequences of $f^-$ and $f^+$ immediately before
and after $f$ in $D$ differ in exactly two adjacent positions
such that in these two positions there is an integer $r$ such that
\[  f^- = (\cdots, r-1,r, \cdots), f = (\cdots, r-1, r+1, \cdots), 
f^+ = (\cdots, r, r+1, \cdots). \]
\end{itemize}
\end{proposition}

In Case 3. we denote the exterior facet by $\facet{f}{\tau}$ where
$\tau$ is the position of the number $r-1$ in $f$.

\begin{proof}
That these cases give exterior facets is immediate as in the 
discussion of the example above. That this is the only way to achieve
exterior facets is also easy to verify.
\end{proof}

\subsection{The supporting hyperplanes}
If $\sigma$ is full dimensional simplicial cone in a vector space $L$,
each facet of $\sigma$ is contained in a unique hyperplane, which is the kernel
of a nonzero linear functional $h : L \pil k$. 

We shall apply this to the cones $\sigma(D)$ in $\Lhkab$, and
find the equations of the hyperplanes $H$ defining the exterior facets of 
$\sigma(D)$. Actually we consider the inclusion  $\sigma(D) \sus
\Lhkab \sus \Dab$ and rather find a hyperplane
$H^\prime$ in $\Dab$ with $H = H^\prime \cap \Lhkab$. 
The equation of such a hyperplane is not unique up to constant however.
Since $\Lhkab$ is cut out by the Herzog-K\"uhl equations,
we may add any linear combinations of these equations, say $\ell$, 
and get a new equation $h^{\prime\prime} = h^\prime + \ell$ defining
another hyperplane $H^{\prime\prime} \sus \Dab$ which 
still intersects $\Lhkab$ in $H$. In Cases 1. and 2. of
Proposition \ref{FacProFacet} there turns out to be a unique
natural choice for the hyperplane, while in Case 3. there
are two distinguished hyperplanes.

\begin{example} \label{ExtEksChains2} We continue Example \ref{ExtEksChains}
and look at the various types of exterior facets of 
Proposition \ref{FacProFacet}.

\begin{itemize}
\item [1.] In the chain
\[ D : (0,1,3) < (0,2,3) < (0,2,4) < (0,3,4) \]
if we look at the facet of $\sigma(D)$ we get by removing $(0,1,3)$,
the natural equation for a hyperplane in $\Dab$ is 
$\beta_{1,1} = 0$. This hyperplane contains $\pi(0,2,3)$, $\pi(0,2,4)$, 
and $\pi(0,3,4)$, but it does not contain $\pi(0,1,3)$. We may get
other equations by adding linear combinations of the Herzog-K\"uhl
equations but this equation is undoubtedly the simplest one.

\item [2.] In the chain 
\[ E : (0,1,3) < (0,1,4) < (0,2,4) < (0,3,4) \]
if we consider the facet of $\sigma(D)$ we get by removing
$(0,2,4)$, the natural equation for a supporting hyperplane is
$\beta_{1,2} = 0$.

\item [3.] In the case that we remove $f = (0,2,4)$ from 
$D$ things are more
refined. There turns out to be two linear functionals on $\Dab$
which define two distinguished hyperplanes, called respectively
the upper and lower hyperplanes. 
We will represent the equation of a hyperplane in $\Dab$
by giving the coefficients of the $\beta_{ij}$. 
To describe the upper hyperplane note that the Betti diagram of the sequence 
$f^+ = (0,3,4)$ immediately after $f$ is

\begin{equation*}
\begin{matrix}
-1 \\ 0 \\ 1 \\ 2 
\end{matrix}
\left [ 
\begin{matrix}
0 & 0 & 0 \\
1^* & 0 & 0 \\
0 & 0^- & 0^- \\
0 & 4^+ & 3^+
\end{matrix} \right ]
\end{equation*}

The nonzero entries of the diagram have been additionally labelled with 
$*,+$ and $+$. Similarly the nonzero
positions of $\pi(f^-)$ will be labelled by $*,-$ and $-$.
Thinking of the Betti diagram as stretching infinitely upwards
and downwards, 
the zeros in the diagram for $f^+$ are divided into an upper and lower
part.
The equation of the
upper hyperplane, the {\it upper equation} will have possible nonzero
values only in the upper part of $f^+$ (marked with normalsized $*$'s):
\begin{center}\label{FacLigHup}
\begin{tabular}{ l | l l l}
-2 &* & * & * \\
-1 & * & * & * \\
0 &  $0^*$  & * & * \\
1 & 0 &  $*^-$  &  $*^-$  \\
2 & 0 &  $0^+$  &  $0^+$  
\end{tabular}
\end{center}
\end{itemize}
\end{example}

\begin{remark}
The choice of facet equation $\beta_{ij} = 0$ for exterior facets
of type 1 and 2 and of the upper equation for exterior facets of type
3 is further justified in the last paragraph of Subsection \ref{SubsekFurNC}.
The Betti diagrams of {\it all} graded modules whose 
Betti diagram is in the window $\Dab$, generate a full-dimensional cone
in this window.
The exterior facet types above have corresponding larger facets in this
cone, and the equations above give the unique (up to scalar) equations
of these larger facets.
\end{remark}

Before proceeding to find the upper hyperplane equation, we note
the following which says that the choice of window bounds $\bfa$ and 
$\bfb$ does not have any essential effect on the exterior facets, and
that the exterior facets of type 3 essentially only depend on
the $f$ omitted and not on the chain.

\begin{lemma} \label{FacetLemExtend}
Consider facets of type 3 in Proposition \ref{FacProFacet}.
\begin{itemize}
\item[a.] If $D$ and $E$ are two maximal chains in $[\bfa, \bfb]_\inc$
which both contain the subsequence $f^- < f < f^+$, the exterior facets
$\sigma(D \backslash \{f \})$ and $\sigma(E \backslash \{ f \})$
define the same hyperplane in $\Lhkab$. 
\item[b.] Let $\bfa^\prime \leq \bfa \leq \bfb \leq \bfb^\prime$ and
suppose $D^\prime$ is a maximal chain in $[\bfa^\prime, \bfb^\prime]_\inc$
restricting to $D$ in $[\bfa, \bfb]_\inc$. If $H^\prime$ in 
${\mathbb D}(\bfa^\prime, 
\bfb^\prime)$ is a hyperplane defining $\sigma(D^\prime \backslash \{ f \})$,
then $H^\prime \cap \Dab$ is a hyperplane defining 
$\sigma(D \backslash \{ f \})$.
\end{itemize}
\end{lemma}

\begin{proof}
As for part b. $\sigma(D \backslash \{ f \})$ is a subset of 
$\sigma(D^\prime \backslash \{ f \})$ and so is contained in $H^\prime$.
But $H^\prime$ does not contain $\Lhkab$ since $H^\prime$
does not contain $\pi(f)$ which is contained in this linear space, 
and so $H^\prime  \cap \Lhkab$ is a hyperplane in
$\Lhkab$. Thus $H^\prime \cap \Dab$ is a hyperplane in $\Dab$
defining $\sigma(D \backslash \{f \})$.

For part a. note that if $D$ is the chain
\[ \bfa = \bfd^1 < \cdots < \bfd^{p-1} = f^- < f < f^+ = \bfd^{p+1} < 
\cdots < \bfd^r = \bfb \]
the space $L^-$ spanned by $\pi(\bfd^1), \ldots, \pi(\bfd^{p-1})$ 
is by Proposition \ref{BSTProLin} equal to $L^{HK}(\bfa, f^-)$, and
so depends only on $\bfa$ and $f^-$. Similarly $L^+$ spanned by 
$\pi(\bfd^{p+1}), \ldots, \pi(\bfd^r)$ is equal to $L^{HK}(f^+, \bfb)$.
The hyperplane of $\sigma(D \backslash \{ f \})$ is then spanned
by $L^-$ and $L^+$, which depends only on $f$, $f^+$, $f^-$, $\bfa$ and $\bfb$.
\end{proof}

\begin{example} \label{FacExUp}
Let us return to Example \ref{ExtEksChains2} to find the hyperplane equations
when we remove $f = (0,2,4)$ from $D$. By the previous proposition
we may as well assume that $\bfa$ is some tuple with small coordinates
and $\bfb$ is a tuple with large coordinates. 

The upper hyperplane equation $\hup$, which has the form given
in (\ref{FacLigHup}),
does not vanish on $\pi(f)$ but will, by Lemma 
\ref{FacetLemExtend} vanish on $\pi(g)$ when $g < f$. 
In particular it vanishes on 
\[ \pi(f^-) = \pi(0,2,3) = 
\begin{matrix}
-1 \\ 0 \\ 1
\end{matrix}
\left [ 
\begin{matrix}
 0 & 0 & 0 \\
 1^* & 0 & 0 \\
 0 & 3^- & 2^- 
\end{matrix} \right ],
\]
and so the coefficients of $\hup$ must have the form

\begin{center}
 \begin{tabular}{ l | l l l} 
-2 & * & * & * \\
-1 & * & * & * \\
0 & 0 & * & * \\
1 & 0 & 2$\alpha$ & -3$\alpha$
\end{tabular}
\end{center}
where $\alpha$ is some nonzero constant, which we may as well take to be
$\alpha = 1$. Also $\hup$ must vanish on 
\[  \pi(0,1,3) =
\begin{matrix}
0 \\ 1 \end{matrix}
\left [ 
\begin{matrix}
2 & 3 & 0 \\
0 & 0 & 1
\end{matrix} \right ]. 
\]

Thus shows that the coefficients of $\hup$ must be 

\begin{center}
\begin{tabular}{ l | l l l}
-2 & * & * & * \\
-1 & * & * & * \\
0 & 0 & 1 & * \\
1 & 0 & 2 & -3
\end{tabular} .
\end{center}

We may continue with an element just before $(0,1,3)$ in a maximal chain,
say $(0,1,2)$. Since 
\[  \pi(0,1,2) = 
0 
\left [ \begin{matrix} 
1 & 2 & 1 
\end{matrix} \right ]. 
\]
we get that the coefficients of $\hup$ are 
\begin{center}
\begin{tabular}{ l | l l l}
-2 & * & * & * \\
-1 & * & * & * \\
0 & 0 & 1 & -2\\
1 & 0 & 2 & -3
\end{tabular} .
\end{center}

In this way we may continue and $\hup$ will be uniquely determined
in all positions in the window determined by $\bfa$ and $\bfb$. We
find that the coefficients of $\hup$ is given by the diagram:

\begin{center}
\begin{tabular}{ l | l l l}
-3 & 3 & -2 & 1 \\
-2 & 2 & -1 & 0 \\
-1 & 1 & 0 & -1 \\
0 & $0^*$ & 1 & -2 \\
1 & 0 & $2^-$ & -$3^-$ \\
2 & 0 & $0^+$ & $0^+$.
\end{tabular} 
\end{center}

In order to find the lower equation, we may 
in a similar way consider the diagram of $\pi(f^-)$ 
\begin{equation*}
\begin{matrix}
-1 \\ 0 \\ 1 \\ 2
\end{matrix}
\left [ 
\begin{matrix}
0 & 0 & 0 \\
 1^*  & 0 & 0 \\
0 &  3^-  &  2^-  \\
0 &  0^+  &  0^+  
\end{matrix} \right ].
\end{equation*}
Again thinking of the Betti diagram as stretching infinitely
upwards and downwards, 
the positions with zero are divided into an upper and a lower part. 
There is a unique hyperplane defined by a linear form $\hlow$ which may
have nonzero entries only in the lower part of the diagram of $\pi(f^-)$. 
We find that the coefficients of $\hlow$ are given by the following.

\begin{center}
\begin{tabular}{ l | l l l} 
0 & $0^*$ & 0 & 0 \\
1 & -1 & $0^-$ & $0^-$ \\
2 & -2 & $3^+$ & $-4^+$ \\
3 & -3 & 4 & -5 \\
4 & -4 & 5 & -6 \\
5 & -5 & 6 & -7 
\end{tabular} .
\end{center}
\end{example}

\begin{proposition} \label{FacProUp}
Let $f^- < f < f^+$ be the degree sequence as in part 3 of Proposition
\ref{FacProFacet}. There is a unique hyperplane in $\Dab$, the {\it upper
hyperplane}, that  
contains $\facet{f}{\tau}$ and whose equation has coefficient zero 
of $\beta_{i,j}$ for all $j \geq f_i^+$. 
\end{proposition}

\begin{proof} This is done as in the example by choosing any chain
$f^- =\bfd^{p-1} > \bfd^{p-2} > \cdots > \bfd^1 = \bfa$ and making the
equation of the hyperplane 
vanish on the elements of this chain. 
Lemma \ref{FacetLemExtend} shows that we get the same hyperplane equation
independent of the choice of chain.
\end{proof}

A regular feature of the equations is that the diagonals from lower left to
upper right have the same absolute values but alternating signs in the
range where they are nonzero.

\begin{lemma} Let $b_{ij}$ be the coefficient of $\beta_{ij}$ 
in the upper equation $\hup$. 
If $j < f_i^+$ then $b_{i+1, j} = - b_{i,j}$. 
\end{lemma}

\begin{proof} Both $\hup$ and $\hlow$ are equations of the same
hyperplane in the subspace $L^{HK}(\bfa, \bfb)$.
A linear combination of them, in our examples
$\hup + \hlow$, then vanishes on this space
and so must be a linear combination of the Herzog-K\"uhl equations
(\ref{BSTLigHK}).
But looking at these equations we see that the coefficient of 
$\beta_{i+1, j}$ and $\beta_{i,j}$ always have the same absolute value
but different signs.
\end{proof}

What are the explicit forms
of the facet equations, i.e. what determines the numbers occurring
in these equations?
We are interested in this because each supporting hyperplane $H$ defines a 
halfspace $H^+$ and the intersection of all these halfspaces is a positive
cone contained in  the Boij-S\"oderberg fan $\Sigab$. We will be
able to show that each Betti diagram of a module is in all the positive
halfspaces. This shows that the positive cone $\Pab$ is contained
in the realization of $\Sigab$, so we obtain part b. of 
Theorem \ref{ResTheBS2}.

The numbers in the example above are too simple to make any deductions as to 
what governs them in general. A more sophisticated example is the following.

\begin{example}
The upper equation of $\facet{(-1,0,2,3)}{1}$ has coefficients:
\begin{center}
$U$ :  \begin{tabular}{c | c c c c} 
-4 & 21 & -12 & 5 & 0 \\
-3 & 12 & -5 & 0 & 3 \\
-2 & 5 & 0 & -3 & 4 \\
-1 & $0^*$ & $3^-$ & $-4^-$ & 3 \\
0  & 0 & $0^+$ & $0^+$ & $0^*$ 
\end{tabular}
\end{center}
\end{example}
\noindent where in $U$ the superscripts $*$ and $+$ indicate the nonzero parts
of $\pi(f^+)$, while the $*$ and $-$ indicate the nonzero parts
of $\pi(f^-)$. 
The polynomial ring in this case is $S = \kr[x_1, x_2, x_3]$. 
Eisenbud and Schreyer, \cite{ES}, 
recognised the numbers in this diagram as the Hilbert functions of the homology 
modules of the complex
\begin{equation} \label{ExtLigE} 
E : E^0 = S(1)^5 \mto{d} S(2)^3 = E^1
\end{equation}
for a general map $d$.

The homology table of this complex is : 
\begin{center} \label{ExtLigH}
\begin{tabular}{r|lll lll lll l}
$d$                  & -3 & -2 & - 1 &  0 &  1 &  2 &  3 &  4 &  5 &  6 \\   
\hline
$\dim_\kr (H^1 E)_d$ & 0 & 3 & 4 & 3 & 0 & 0 & 0 & 0 & 0 & 0 \\
$\dim_\kr (H^0 E)_d$ & 0 & 0 & 0 & 0 & 0 & 5 & 12 & 21 & 32 & 45.
\end{tabular}
\end{center}
Two features of this cohomology table that we note are the following.
The dimensions of $(H^0 E)_d$ are the values of the Hilbert polynomial
of $H^0 E$ for $d \geq 1$. This Hilbert polynomial is
\[ 5 \binom{d+3}{2} - 3 \binom{d+4}{2} = (d-1)(d+3). \]
This polynomial also gives the dimensions of $H^1 E$ in the degrees
$d = -2, -1, 0$ but with opposite sign. Note also that the roots
of this polynomial are $1$ and $-3$ which are the negatives of the
first and last entry in the degree sequence $(-1, 0,2,3)$ that
we consider. In fact the lower and upper facet equations are now
fairly simple to describe.

Given a sequence $\bfz :  z_1 > z_2 > \cdots > z_{c-1}$ of integers.
It gives a polynomial 
\[ p(d) = \prod_{i = 1}^{c-1} (d- z_i). \]
Let $H(\bfz)$ be the diagram in ${\mathbb D}$ such that:
\begin{itemize} 
\item The value in position $(0,d)$ is $p(-d)$.
\item The entries in positions $(i+1, d)$ and $(i,d)$ for $ i = 0, \ldots, c-1$
have the same absolute values but opposite signs.
\end{itemize}

\begin{example}
When $c = 3$ and $p(d) = (d-1)(d+3)$ the diagram $H(\bfz)$ is the following 
rotated $90^\circ$ counterclockwise
\begin{center} 
\begin{tabular}{lll lll lll lll}
& 6 & 5 & 4  & 3 & 2 & 1 &  0 &  -1 &  -2 &  -3&   \\
\hline
$\cdots$ &5 & 0 & -3  & -4& -3   & 0 & 5   & 12 & 21  &  32 &   $\cdots$      \\
 $\cdots$ & -12 & -5 & 0 & 3 & 4    & 3 & 0   & -5 & -12 & -21 &   $\cdots$ \\
$\cdots$ & 21 & 12 & 5 & 0 & -3   & -4 & -3 & 0  & 5   & 12 & $\cdots$ \\
$\cdots$ & -32 & -21 & -12 & -5 & 0 & 3 & 4 & 3 & 0 & -5 & $\cdots$ 
\end{tabular}
\end{center}
\end{example}

Now given a $\facet{f}{\tau}$ of the \BS-fan. Associated to the
sequence of integers
\[ \hat{f} :  - f_0 > -f_1 > \cdots > - f_{\tau - 1} > -f_{\tau + 2}
> \cdots > -f_c \]
we get a  $H(\hat{f})$. 
Let $U(f,\tau)$ be the diagram we get by making all entries of 
$H(\hat{f})$ on and below the positions occupied by $\pi(f^+)$
equal to zero.
Explicitly $U(f, \tau)_{ij} = H(\hat{f})_{ij}$ for 
$j < f_i^+$ and $U(f,\tau)_{ij} = 0$ 
otherwise. The associated linear form is then:
\begin{equation} \label{FacLigU}
 \sum_{\stackrel{i < \tau}{d < f_i}} (-1)^i \beta_{i,d} p(-d)
+ \sum_{\stackrel{i = \tau}{d \leq f_\tau}}  (-1)^i \beta_{i,d} p(-d)
+ \sum_{\stackrel{i > \tau}{d < f_i}}  (-1)^i \beta_{i,d} p(-d).
\end{equation}

\begin{proposition} \label{FacProUplig} 
The upper equation $\hup$ of {\bf facet}$(f,\tau)$
has coefficients given by the diagram
$U(f, \tau)$.
The coefficients of the lower facet equation is $H(\hat{f}) - U(f,\tau)$.
\end{proposition}

\begin{proof}
First note that  $\hup(\pi(f))$ is nonzero. If the degree sequence
$f^\prime \geq f^+$
then clearly $\hup(\pi(f^\prime)) = 0$. When $f^\prime \leq f^-$ it is
shown in \cite[Theorem 7.1]{ES} that $\hup(\pi(f^\prime)) = 0$.
\end{proof}

\subsection{Pairings of vector bundles and resolutions} 

In order to prove Proposition \ref{FacProUplig} 
we had to show that the hyperplane
equation $\hup$ given by $U(f,\tau)$ is positive on $\pi(f)$ and vanishes on the
other $\pi(f^\prime)$. 
With the explicit forms we have for all these expressions this could
be done with numerical calculations. However to prove Theorem \ref{ResTheBS2}
we need to show that the form given by $\hup$ is non-negative on all
Betti diagrams of Cohen-Macaulay modules.

In order to prove this positivity we must go beyond the numerics.
It then appears that if $\beta$ is a Betti diagram, the linear
functional determined by 
$U(f,\tau)$ evaluated on  $\beta$ 
arises from a pairing between a Betti diagram and the cohomology table of 
a vector bundle. This is the fruitful viewpoint which enables us to show
the desired positivity.

\begin{example}
Going back to the complex (\ref{ExtLigE}), if we sheafify this complex
to get a complex of direct sums of line bundles on the projective
plane $\proj{2}$ 
\[ \tilde{E} : \gO_{\proj{2}}(1)^5 \mto{\tilde{d}} \gO_{\proj{2}}(2)^3, \]
the map $\tilde{d}$ is surjective and so the only nonvanishing homology
is $\gE = 
H^0 (\tilde{E})$. The table below is the cohomology
table of the vector bundle $\gE$.

\begin{center}
\begin{tabular}{r|c l l l l l l l l l l l l l c}
$d$  & $\cdots$  
&-6 & -5 & -4 & -3 & -2 & - 1 &  0 &  1 &  2 &  3 &  4 &  5 &  6 & $\cdots$ \\ 
\hline  
$\dim_\kr H^2 \gE(d)$ & $\cdots$ & 21 & 12 & 5 & 0 & 0 & 0 & 0 & 0 & 0 & 0 & 0 
& 0 & 0 & $\cdots$ \\
$\dim_\kr H^1 \gE(d)$ & $\cdots$ & 0 & 0 & 0 & 0 & 3 & 4 & 3 & 0 & 0 & 0 & 0 
& 0 & 0 & $\cdots$ \\
$\dim_\kr H^0 \gE(d)$ & $\cdots$ & 0 & 0 & 0 & 0 & 0 & 0 & 0 & 0 & 5 & 12 & 21 
& 32 & 45 & $\cdots$.
\end{tabular}
\end{center}
and the values are the absolute values of $(d-1)(d+3)$. This is an example
of a bundle with supernatural cohomology as we now define.
\end{example}

\begin{definition} Let 
\[ z_1 > z_2 > \cdots > z_m \]
be a sequence of integers. A vector bundle $\gE$ on the projective space 
$\proj{m}$ has {\it supernatural cohomology} if:
\begin{itemize}
\item[a.] The Hilbert polynomial is $\chi \gE(d) = \frac{r^0}{m!}
\cdot \prod_{i = 1}^m (d-z_i)$ for a constant $r^0$ which must be the
rank of $\gE$. 
\item[b.] For each $d$ let $i$ be such that $z_i > d > z_{i+1}$.
Then 
\[  H^i \gE(d) = \begin{cases}  \frac{r^0}{m!}
\cdot \prod_{i = 1}^m |d-z_i|, & z_i > d > z_{i+1} \\
 0, & \mbox{otherwise}
\end{cases} \]
\end{itemize}

The sequence $z_1 > z_2 > \cdots > z_m$ is called the {\it root sequence}
of the bundle $\gE$. 

\end{definition}

In particular we see that for each $d$ there is at most one nonvanishing
cohomology group. We show in Section \ref{ExiSek} that 
for any sequence $\bfz$ of strictly decreasing integers such a vector
bundle exists.

\begin{remark} The naturality of the notion of supernatural cohomology
for a vector bundle, may be seen from the fact that it is  
equivalent to its  Tate resolution, see Subsection \ref{SubsekFurPro}, being
pure, i.e. each cohomological term in the Tate resolution, a free module
over the exterior algebra, being generated in a single degree. 
\end{remark}


Proposition \ref{FacProUplig} and the explicit form 
(\ref{FacLigU}) just before it, may now be translated to the following.

\begin{proposition} \label{FacProUpgamma}
For a $\facet{f, \tau}$ let $\gE$ be a vector bundle on
$\proj{c-1}$  with
supernatural cohomology and root sequence 
\[ -f_0 > -f_1 > \cdots > -f_{\tau - 1} > -f_{\tau + 2} > \cdots > -f_c. \]
Let $\gamma_{i, d} = H^i \gE (d)$ and $\gamma_{\leq i, d}$ the alternating
sum $\gamma_{0,d} - \gamma_{1,d} + \cdots + (-1)^i \gamma_{i,d}$. 
 The upper facet equation $\up{f}{\tau}(\beta)$ is defined by the linear
form
\begin{eqnarray*}
  & 
\underset{\stackrel{i < \tau }{d < f_i}}\sum
(-1)^{i} \beta_{i,d} \gamma_{\leq i, -d}  & 
+   
\underset{{d \leq  f_\tau}}{\sum}
(-1)^{\tau} \beta_{\tau,d} \gamma_{\leq \tau, -d} 
+ \underset{{d <   f_{\tau+1}}}{\sum}
(-1)^{\tau + 1} \beta_{\tau+1,d} \gamma_{\leq \tau, -d} 
\\ 
+ & \underset{\stackrel{i > \tau + 1}{d < f_i}}{\sum}
(-1)^{i} \beta_{i,d} \gamma_{\leq i-2, -d} &. 
\end{eqnarray*}
\end{proposition}

To understand this as a special case of the upcoming
(\ref{FacLigParing2}), we may note the following.
\begin{itemize}
\item $\gamma_{\leq i, -d} = 0$ when $i < \tau$ and  $d \geq f_i$.
\item $\gamma_{\leq \tau - 1, -d} = 0$ when $d > f_\tau$.
\item  $\gamma_{\leq i-2, -d} = 0$ when $i > \tau + 1$ and $d \geq f_i$.
\end{itemize}

By studying many examples and a leap of insight, Eisenbud and 
Schreyer defined for any integer $e$ and $0 \leq \tau \leq n-1$
a pairing $\langle \beta, 
\gamma \rangle_{e,\tau}$ between diagrams and cohomology tables
as the expression
\begin{eqnarray} \notag
& \underset{i < \tau, d \in {\hele}}\sum (-1)^i \beta_{i,d} \gamma_{\leq i, -d} & \\
\label{FacLigParing2}
+ & \underset{d \leq  e}\sum (-1)^\tau \beta_{\tau,d} \gamma_{\leq \tau, -d} & 
+ \underset{d > e}\sum (-1)^\tau \beta_{\tau,d} \gamma_{\leq \tau-1, -d} \\
\notag + & \underset{d \leq   e+1}\sum (-1)^{\tau + 1} \beta_{\tau+1,d} 
\gamma_{\leq \tau, -d} & 
+ \underset{d > e+1}\sum (-1)^{\tau + 1} 
\beta_{\tau+1,d} \gamma_{\leq \tau-1, -d} \\
\notag 
+ & \underset{i > \tau+1, d \in {\hele}}\sum (-1)^i \beta_{i,d} 
\gamma_{\leq i-2, -d} 
\end{eqnarray}
When $e = f_\tau$ and $\gamma$ is the cohomology table of the 
supernatural bundle of Proposition \ref{FacProUpgamma}, this reduces to the 
expression given there.
If $\Fd$ is a resolution 
and $\gF$ is a coherent sheaf  on $\proj{n-1}$
we let $\gamma(\gF)$ be its cohomology table and define
\[ \langle \Fd, \gF \rangle_{e, \tau} 
= \langle \beta(\Fd), \gamma(\gF) \rangle_{e,\tau}. \]


That this pairing is the natural one is established by
the following which is the key result of the paper \cite{ES},
extended somewhat in \cite{ES2}.

\begin{theorem} \label{FacetThePairing}
For any {\it minimal} free resolution $\Fd$ of length $\leq c$ 
and coherent sheaf $\gF$ on $\proj{c-1}$
the pairing
\[  \langle \Fd, \gF \rangle_{e, \tau} \geq 0. \]
\end{theorem}

The proofs of this uses the spectral sequence of a double 
complex. It is not long but somewhat technical so we do not reproduce
it here, but refer the reader to Theorem 4.1 of \cite{ES} and
Theorem 4.1 of \cite{ES2}, or the latest Theorem 3.3 of
\cite{Sch}. It is essential that $\Fd$ is a {\it minimal} free resolution. 
Using the above results we are now in a position to prove Theorem 
\ref{ResTheBS2} or equivalently Theorem \ref{ResTheGeo} b.

\begin{proof}
If {\bf facet}$(f, \tau)$ is a facet
of type 3 of the fan 
$\Sigma(\bfa, \bfb)$, the upper
hyperplane equation is $\langle -, \gamma(\gE) \rangle_{e,\tau} = 0$
where $e = f_\tau$ and $\gE$ given in Proposition \ref{FacProUpgamma}.
For facets of type 1 or 2 the hyperplane equations 
are $\beta_{ij} =  0$ for suitable $i,j$. 

Each exterior facet determines a non-negative half plane $H^+$. Since 
the forms above are non-negative
on all Betti diagrams $\beta(M)$ in $\Dab$ by
Theorem \ref{FacetThePairing}, the cone $B(\bfa, \bfb)$
is contained in the intersection of all the half planes $H^+$
which again is contained in the 
fan $\Sigma(\bfa, \bfb)$. 
\end{proof}

\section{The existence of pure free resolutions and of 
vector bundles with supernatural cohomology}
\label{ExiSek}

There are three main constructions of pure free resolutions. The
first appeared on the arXiv.org in September 2007 \cite{EFW}. This
construction works in char  $\kr  = 0$ and is 
the $GL(n)$-equivariant resolution. Then in December 2007 appeared 
the simpler but rougher construction of \cite{ES} which works in all 
characteristics. In the paper \cite{EFW} there also appeared
another construction, resolutions of modules supported on 
determinantal loci.
This construction is somewhat less celebrated but certainly deserves
more attention for its naturality and beauty.
It is a comprehensive generalization of the Eagon-Northcott
complexes and Buchsbaum-Rim complexes in a generic setting.

Quite parallel to the first two constructions of pure free resolutions,
there are analogous constructions of vector bundles on $\proj{n-1}$ 
with supernatural cohomology. These constructions are actually simpler
than the constructions of pure free resolutions, and were to some
extent known before the term supernatural cohomology was coined in 
\cite{ES}. In the following we let $V$ be a finite dimensional 
vector space and let  $S$ be the symmetric algebra $S(V)$ with unique graded 
maximal ideal $\mm$.

\subsection{The equivariant pure free resolution}
\label{SubsekExiEkvi}
We shall first give the construction of the $GL(V)$-equivariant pure resolution
of type $(1, 1, \ldots, 1)$ and more generally of type 
$(r, 1, \ldots, 1)$ for $r \geq 1$. These cases are known classically,
and provide the hint for how to search for equivariant pure resolutions
of any type $\bfd$. 

\subsubsection{Pure resolutions of type $(1, 1, \ldots, 1)$.} 
In this case the resolution
is the Koszul complex
\[ S \vpil S \tek V \vpil S \tek \wedge^2 V \vpil \cdots \vpil 
S \tek \wedge^n V \]
which is a resolution of the module  $\kr = S/\mm$.
(We consider $V$ to have degree $1$, and $\wedge^p V$ to have
degree $p$.)
The general linear linear group $GL(V)$ acts on each term $S \tek \wedge^p V$
since it acts on $S$ and $\wedge^p V$. And the differentials
respect this action so they are maps of $GL(V)$-modules. We say
the resolution is $GL(V)$-equivariant.
To define the differentials
note that there are $GL(V)$-equivariant maps 
$\wedge^{p+1} V \mto{\rho} V \tek \wedge^p V$. Also let 
$\mu :  S \tek V \pil S$ be the multiplication map.
The differential in the Koszul complex is then:
\[ S \tek \wedge^{p+1} V \mto{\bfen_S \te \rho }
S \tek V \tek \wedge^p V 
\mto{\mu \te \bfen} S \tek \wedge^p V. \]

\subsubsection{Pure resolutions of type $(r, 1, \ldots,1)$. }
Let us consider resolutions of type $(3,1,1)$. 

\begin{example} \label{ExiEksMr}
In $S = \kr[x_1, x_2, x_3]$ the ideal 
$\mm^3 = (x_1, x_2, x_3)^3$ has $3$-linear resolution. 
The resolution of the quotient ring (a Cohen-Macaulay module
of codimension three) is :
\[ S \vmto{d_0} S(-3)^{10} \vmto{d_1} S(-4)^{15} \vmto{d_2} S(-5)^6. \]


Looking at this complex in degree $3$, the exponent $10$ is
the third symmetric power $S_3(V)$ of $V = \langle x_1, x_2, x_3 \rangle$. 
Looking at the complex in degree $4$, we see that $15$ is the dimension 
of the kernel of the multiplication $V \tek S_3(V) \pil S_4(V)$.
This map is $GL(V)$-equivariant and the kernel is the representation 
$S_{3,1}(V)$ which has dimension $15$ (see below for references explaining
this representation). The inclusion 
\[ S_{3,1}(V) \pil V \tek S_3(V) \] induces a composition
\[ V \tek S_{3,1}(V) \pil V \tek V \tek S_3(V) \pil S_2(V) \tek S_3(V). \]
This is the map $d_1$ in degree $5$ and 
the kernel of this map is the representation $S_{3,1,1}(V)$ whose dimension
is $6$, accounting for the last term in the resolution above.
\end{example}

In general it is classically known that the 
resolution of $S/\mm^r$ is
\begin{equation} \label{ExiLigSr}
S \vpil S \tek S_r(V) \vpil S \tek S_{r,1}(V)  \vpil \cdots
\vpil S \tek S_{r,1^{n-1}}(V).
\end{equation}
This is a pure resolution of type $(r,1,1,\ldots,1)$. 

\subsubsection{Representations of $GL(V)$.}
Let us pause to give a brief explanation of the terms $ S_{r,1^{n-1}}(V)$.
The irreducible representations of $GL(V)$ where $n = \dim_\kr V$ 
are classified by partitions of integers
\[ \lambda \, : \, \lambda_1 \geq \lambda_2 \geq \cdots \geq \lambda_n. \]
For each such partition there is a representation denoted by
$S_\lambda(V)$. The details of the construction are easily found in various
textbooks like \cite{FuHa}, \cite{Ei} or \cite{We}.

Such a partition may be displayed in a Young diagram if 
$\lambda_n \geq 0$.  With row index
going downwards, put $\lambda_i$ boxes in row $i$, and align the
rows to the left. (Call this Horizontal display. Another 
convention is to display $\lambda_i$ boxes in column $i$ and top align 
them. Call this Vertical display.)

If
\[ \lambda^\prime \, : \,  \lambda_1  + a \geq \lambda_2 + a \geq \cdots \geq 
\lambda_n + a, \]
then $S_{\lambda^\prime}(V) = (\wedge^n V)^a \tek S_{\lambda}(V)$
where $\wedge^n V$ is the one-dimensional determinant representation.
In Example \ref{ExiEksMr} we needed to consider tensor products
$S_\lambda(V) \tek S_\mu(V)$. 
In char $\kr = 0$ this decomposes
into a direct sum of irreducible representations. In general this is
complicated, but in one case there is a simple rule which is of central
use for us in the construction of equivariant resolutions.

For two partitions $\mu$ and $\lambda$ with $\mu_i \geq \lambda_i$ for
each $i$, say that $\mu \backslash \lambda$ is a {\it horizontal strip}
({\it vertical strip} with Vertical display) 
if $\mu_i \leq \lambda_{i-1}$ for all $i$. Thus when removing the
diagram of $\lambda$ from that of $\mu$, no two boxes are in the
same column. 

\begin{pieri}
\[ S_r(V) \tek S_{\lambda}(V) = 
\underset{\stackrel{\mu \backslash \lambda}{
\stackrel{\mbox{is a horizontal strip}} {\mbox{with $r$ boxes}}}}
{\bigoplus} S_{\mu}(V).\]
\end{pieri}

\subsubsection{Resolutions of length two.}
Let us examine one more case where equivariant pure free resolutions are easily 
constructed: The case when $S = \kr[x,y]$.

\begin{example}
\[ S^2 \xleftarrow { \left [ \begin{matrix} 4x^3 &  3x^2y & 2xy^2 & y^3 & 0 \\ 
                                     0 & x^3 & 2x^2y & 3xy^2 & 4y^3 
\end{matrix} \right ]}
S(-3)^5
\xleftarrow{ \left [ \begin{matrix} y^2 &  0 &  0  \\ 
                                  -2xy &  y^2 & 0 \\ 
                                  x^2 & -2xy & y^2 \\
                                  0  &  x^2 & -2xy \\
                                 0 &  0 & x^2
\end{matrix} \right ]}
S(-5)^3 \]

The matrices here are chosen so that the complex is 
$GL(2)$-equivariant, but we are really free to vary the
coefficients of the first matrix as we like, in an open set,
and there will be a suitable match for the second matrix.
\end{example}

In general one may construct a $GL(2)$-equivariant complex

\begin{equation} \label{ExiLigSab} 
S \tek S_{a-1,0} \vpil S \tek S_{a+b-1,0} \vpil S \tek S_{a+b-1,a}. 
\end{equation}
This is a resolution of type $(a-1,a+b-1,2a+b-1)$. By twisting it with
$a-1$ it becomes of type $(0,b,a+b)$. 

\subsubsection{Resolutions of length three.}
Now suppose we want to construct pure resolutions of type 
$(d_0, d_1, d_2, d_3)$. Since we may twist the complex by $-r$ to get
a pure resolution of type $(d_0 + r, d_1 + r, d_2 + r, d_3 + r)$, what
really matters are the differences $e_1 = d_1 - d_0, e_2 = d_2 - d_1$, and
$e_3 = d_3 - d_2$. 
Looking at the complexes
(\ref{ExiLigSr}) and (\ref{ExiLigSab}) it seems that one should
try to construct a complex as follows
\begin{eqnarray} \label{ExiLigSla} 
S \tek S_{\lambda_1, \lambda_2, \lambda_3} & \vmto{d_1} &  
S \tek S_{\lambda_1 + e_1, \lambda_2, \lambda_3} \\ \notag
& \vmto{d_2} & S \tek S_{\lambda_1 + e_1, \lambda_2 + e_2, \lambda_3}
\vmto{d_3}  S \tek S_{\lambda_1 + e_1, \lambda_2 + e_2, \lambda_3+ e_3}. 
\end{eqnarray}
After looking at the numerics, i.e. the dimensions of the representations
and the Herzog-K\"uhl equations, there is one choice that fits exactly.
This is taking 
\[ \lambda_3 = 0, \quad \lambda_2 = e_3 -1, 
\quad \lambda_1 = (e_2 - 1) + (e_3 - 1).\]
We must then construct these complexes. To construct $d_1$ one must
chose a map
\[ S_{\la_1 + e_1, \la_2, \la_3} \pil S_{e_1} \tek S_{\la_1, \la_2, \la_3}. \]
But by Pieri's rule the first module occurs exactly once as a component
in the second tensor product. Hence there is a nonzero map as above,
unique up to a nonzero constant.  Similarly $d_2$ is given by 
\[S_{\la_1 + e_1, \la_2 + e_2, \la_3} \pil S_{e_2} \tek 
S_{\la_1 + e_1, \la_2, \la_3} \]
and again by Pieri's rule there is a nonzero such map unique up to 
a nonzero constant.  Similarly for $d_3$. Hence up to multiplying
the differentials with constants there is a unique possible such complex
(\ref{ExiLigSla}) with nonzero differentials. 
What must be demonstrated is that this is a resolution,
i.e. the only homology is the cokernel of $d_1$. And in fact this is
the challenging part.

\subsubsection{General construction of equivariant resolutions.}
To construct a pure resolution of type $(d_0,d_1, \ldots, d_n)$
in general one lets $e_i = d_i - d_{i-1}$. Let 
$\lambda_i = \sum_{j > i} (e_j - 1)$ and define the partition
\[\alpha(\bfe, i) : \la_1 + e_1, \ldots, \la_i + e_i, \la_{i+1}, \ldots, 
\la_n. \] 
\begin{theorem}[\cite{EFW}]
There is a $GL(n)$-equivariant resolution
\[ E(\bfe) : S \tek S_{\alpha(\bfe,0)} \vpil 
 S \tek S_{\alpha(\bfe,1)} \vpil \cdots \vpil
 S \tek S_{\alpha(\bfe,n)}. \]
This complex is uniquely defined up to multiplying the differentials
by nonzero constants.
\end{theorem}

\medskip

 These equivariant complexes have a canonical position as follows. Since
the complex above is equivariant for $GL(n)$ it is equivariant for the 
diagonal matrices in $GL(n)$. Hence it is a $\hele^n$-graded complex.
Fix a sequence of differences $(e_1, \ldots, e_n)$. 
Consider $\hele^n$-graded resolutions
\begin{equation}  \label{ExiLigRes}
F_0 \vpil F_1 \vpil \cdots \vpil F_n 
\end{equation}
of artinian $\hele^n$-graded modules which i) become pure when taking
total degrees, i.e. when making a new grading by the map $\hele^n \pil \hele$
given by $(a_1, \ldots, a_n) \pil \sum_1^n a_i$, and ii) such that
the differences of these total degrees are the fixed numbers
$e_1, \ldots, e_n$. 
Each $F_i = \oplus_{\bfj \in \hele^{n}} S(-\bfj)^{\beta_{i\bfj}}$. We may encode
the information of all the multigraded Betti numbers $\beta_{i\bfj}$ as
an element in the Laurent polynomial ring
$T = \hele[t_1, t_1^{-1}, \ldots, t_n, t_n^{-1}]$. Namely for $i = 0, \ldots, n$
let 
$B_i = \oplus_{\bfj \in \hele^{n}} \beta_{i\bfj} t^{\bfj}$,
which we call the Betti polynomials. Consider the lattice ($\hele$-submodule)
$L$ of $T^{n+1}$ generated by all tuples of Betti polynomials
$(B_0, \ldots, B_n)$ derived from resolutions (\ref{ExiLigRes}). 
This is in fact a $T$-submodule of $T^{n+1}$.

The Betti polynomial of the module $S \tek S_{\la}(V)$ is the Schur polynomial
$s_{\la}$ and so the tuple of the equivariant resolutions $E(\bfe)$ is 
\[ s(\bfe)  = (s_{\alpha(\bfe,0)}, s_{\alpha(\bfe,1)}, \ldots, s_{\alpha(\bfe, n)}). \]
Fl{\o}ystad shows that this tuple has a distinguished status
among tuples of Betti polynomials of $\hele^n$-graded resolutions 
of artinian modules.
\begin{theorem}[Theorem 1.2, \cite{Fl}] \label{ExiTheSchurArtin} 
Let char $\kr = 0$ and assume the 
greatest common divisor of $e_1, \ldots, e_n$ is $1$. 
The $T$-submodule $L$ 
of $T^{n+1}$ 
is a free $T$-module of rank one. The tuple 
$s(\bfe)$ is, 
up to a unit in $T$ (which is $\pm t^{\bfa}$, where $t^{\bfa}$ is a 
Laurent monomial), the unique generator of this $T$-module.
\end{theorem}

\subsubsection{Generalizations} \label{ExiSssGen}
The diagram $\alpha(\bfe, 1) \backslash \alpha(\bfe, 0)$ is a horizontal
strip living only in the first row.
In \cite{SW}, S.Sam and J.Weyman consider partitions $\beta$ and $\alpha$
such that $\beta \backslash \alpha$ is a horizontal strip
(vertical strip in Vertical display). They give explicitly the minimal
free resolutions \cite[Thm. 2.8]{SW},  of the cokernel of the map 
(char $\kr = 0$):
\[ S \tek S_\beta(V) \pil S \tek S_{\alpha}(V). \]
This cokernel may no longer be a Cohen-Macaulay module. It may even
have positive rank. In the case that $\beta \backslash \alpha$ contains
boxes only in the $i$'th and $n$'th row for some $i$, they show
that the resolution is pure, \cite[Cor.2.11]{SW}.
The methods used in this paper have the advantage that they are more
direct and explicit than the inductive arguments given in 
\cite{EFW}.

More generally they give the (not necessarily minimal) free 
resolution of the cokernel of
\[ \bigoplus_i S \tek S_{\beta^i} (V) \pil S \tek S_{\alpha}(V) \]
where $\beta^i \backslash \alpha$ are horizontal strips, 
and if for each $i$ the horizontal strip lives in one row, the
resolution is minimal. 

Sam and Weyman in \cite[Section 3]{SW} also generalize the construction of 
$GL(V)$-equivariant pure resolutions to resolutions equivariant for
the symplectic and orthogonal groups.

\subsection{Equivariant supernatural bundles}
\label{SubsekExiEkvisup}
The equivariant resolution has an analog in the construction of bundles
with supernatural cohomology.
Given any ring $R$ and an $R$-module $F$,
one may for any partition $\lambda$ in a functorial way construct the 
Schur module $S_\lambda F$,
see \cite{Ei}, \cite{FuHa} or \cite{We}. In the case when $R = \kr$
and $F$ is a vector space
$V$ in char $\kr = 0$, with $GL(V)$ acting, 
the Schur modules $S_\lambda V$ give the irreducible
representations of $GL(V)$. The construction of $S_\lambda F$ respects 
localization and so for a locally free sheaf $\gE$, an algebraic vector bundle
on a scheme, 
we get Schur bundles
$S_\lambda \gE$. In particular consider the sheaf of differentials
on $\proj{n}$, the kernel of the natural map {\it ev}:
\[ 0 \vpil \opn 
\vmto{ev}  H^0 \opn(1) \tek \opn(-1) \vpil \Omega_{\proj{n}} \vpil 0. \]
We may construct Schur bundles $S_\lambda (\ompn(1))$. 

\begin{example}
The cohomology of the bundle $\ompn(1)$ is well known and it has supernatural
cohomology. It is easily computed by the long exact cohomology sequence
associated to the short exact sequence above. 
\begin{itemize}
\item In the range $1 \leq p \leq n-1$ the only 
nonvanishing cohomology $H^p \ompn(i)$ is when $p = 1$ and $i = 0$:
$H^1 \ompn \iso \kr$. 
\item $H^0 \ompn(i)$ vanishes for $i \leq 1$ 
and is nonvanishing for $i \geq 2$.
\item $H^n \ompn (i)$ vanishes for $i \geq -(n-1)$ and is non-vanishing
for $i \leq -n$. 
\end{itemize}

The root sequence of $\ompn(1)$ is $0, -2, -3, \ldots, -n$ and
its Hilbert polynomial is 
\[ \frac{1}{(n-1)!} \cdot z \cdot \prod_{i = 2}^{n} (z+i). \]
\end{example}

In general the bundle $S_\lambda (\ompn(1))$ has supernatural cohomology. 
It is standard to compute its cohomology by the Borel-Bott-Weil formula
in the theory of linear algebraic groups, \cite{Ja} or \cite{We}.
The computation of its cohomology is done explicitly in 
\cite[Section 4]{Fl2}, or in \cite[Theorem 5.6]{ES4} for the dual bundle. 
In fact the nonzero cohomology modules
 $H^p S_\lambda (\ompn(i))$ are all
irreducible representations $S_\mu V$, where $\mu$ depends on $\lambda, i$
and $p$. 

\begin{theorem} The Schur bundle $S_\lambda (\ompn(1))$ has supernatural
cohomology with root sequence 
\[ \lambda_1-1, \lambda_2 - 2,  \lambda_3 - 3, \ldots, \lambda_n - n. \]
\end{theorem}


\subsection{Characteristic free supernatural 
bundles} \label{SubsekExiChfreesup}

It is a general fact that if $\gF$ is a coherent sheaf on $\proj{N}$
and $\proj{N} \overset{\pi}{\dashrightarrow} \proj{n}$ is a projection whose
center of projection is disjoint from the support of $\gF$, then
$\pi_*(\gF)$ and $\gF$ have the same cohomology 
\[  H^i (\proj{n}, (\pi_* \gF) (p)) \iso H^i (\proj{N},\gF (p)) \]
for all $i$ and $p$. (By the projection formula \cite[p.124]{Ha},
and since one can use flasque resolutions to compute cohomology,
\cite[III.2]{Ha}.)

\begin{example}
The Segre embedding embeds $\proj{1} \times \proj{n-1}$ as a variety of 
degree $n$ into $\proj{2n-1}$. If we take a general projection
$\proj{2n-1} \dashrightarrow \proj{n}$, the line bundle 
$\gO_{\proj{1}}(-2) \te \gO_{\proj{n-1}}$ projects down to a vector bundle
of rank $n$ on $\proj{n}$ which is the sheaf of differentials
$\ompn$. In fact the cohomology of the line bundle  
$\gO_{\proj{1}}(-2) \te \gO_{\proj{n-1}}$ and its successive twists by
 $\gO_{\proj{1}} (1) \te \gO_{\proj{n-1}}(1)$
is readily computed by the 
K\"unneth formula and it is a {\it sheaf} with supernatural
cohomology. It has the same cohomology as $\ompn$, and it is 
not difficult to argue that the projection is actually this bundle.
\end{example}

This example may be generalized as follows.
The Segre embedding embeds $\proj{a} \times \proj{b}$ into
$\proj{ab + a+ b}$ as a variety of degree $\binom{a+b}{a}$. 
Consider the line bundle $\gO_{\proj{a}} (-a-1) \te \gO_{\proj{b}}$ 
on $\proj{a} \times \proj{b}$. 
The line bundle of the hyperplane divisor on $\proj{ab + a+ b}$ 
pulls back to $\gO_{\proj{a}} (1) \te \gO_{\proj{b}}(1)$
and by twisting with this line bundle, the above line bundle is
a sheaf with supernatural cohomology. 
Taking a general projection of $\proj{ab + a+b}$ to $\proj{a+b}$ 
this line bundle projects down to the bundle $\wedge^a \Omega_{\proj{a+b}}$ 
of rank $\binom{a+b}{a}$, as may be argued using Tate resolutions,
\cite[Prop. 3.4]{Fl3}.
The root sequence of this bundle is $a, a-1, \ldots, 1, -1, -2, \ldots, -b$. 
It is natural to generalize this by looking at Segre embeddings
$\proj{a_1} \times \cdots \times \proj{a_r} \inpil \proj{N}$ composed with
a general projection $\proj{N} \dashrightarrow \proj{n}$ where
$n = \sum_i a_i$. 

\begin{theorem} 
Let a root sequence be the union of sets of consecutive integers
\[ \bigcup_{i = 1}^r \{ z_i, z_i - 1, \ldots, z_i - (a_i -1) \}\]
where $z_i \geq z_{i+1} + a_i$, and 
let  $n = a_1 + a_2 + \cdots + a_r$. The line bundle
\[ p_1^* \gO_{\proj{a_1}}(-z_1 - 1) \te \cdots  \te p_{r-1}^* 
\gO_{\proj{a_{r-1}}}(-z_{r-1} - 1) \te p_r^* \gO_{\proj{a_r}}(-z_r - 1) \]
considered on the Segre embedding of $\proj{a_1} \times \cdots \times
\proj{a_r}$ has supernatural cohomology. By a general projection down
to $\proj{n}$ it projects down to a vector bundle with supernatural
cohomology of rank $\binom{n}{a_1 a_2 \cdots a_r}$ and root sequence
given above.
\end{theorem}

\begin{remark}
Although in the case $r = 2$ the projection is a Schur bundle $\wedge^a
\ompn$, it is no longer true for $r > 2$ that one gets twists of Schur bundles
$S_\lambda(\ompn)(p)$, as may be seen from the ranks.
\end{remark}

\subsection{The characteristic free pure resolutions}
\label{SubsekExiChfree}
In this construction of \cite[Section 5]{ES}
one starts with a complex of locally free sheaves on
a product of projective spaces $\proj{m_0} \times \proj{m_1} \times
\cdots \times \proj{m_r}$, whose terms are
direct sums of line bundles $\gO(t_0,  \ldots, t_r)$. The complex is linear
in each coordinate twist and is exact except at the start, so is a locally
free resolution. Then we successively push this complex forward by 
omitting one factor in the product of projective spaces at a time.
Each time some linear part of the complex ``collapses'', so that 
at each step we get ``multitwisted'' pure resolutions. In the end
we will have a singly twisted pure resolution on $\proj{m_0}$ of the
type we desire.

The main ingredient in the construction is the following.

\begin{proposition}[Proposition 5.3, \cite{ES}] \label{ExiProProj}
Let ${\mathcal F}$
be a sheaf on $X \times {\mathbb P}^m$, and denote by $p_1$ and $p_2$
the projections onto the factors of this product. Suppose ${\mathcal F}$ 
has a resolution
\[  p_1^*\gG_0 \sqte p_2^*\opm(- e_0) \vpil \cdots \vpil
p_1^*\gG_N \sqte p_2^*\opm(-e_N) \vpil 0. \]
where $e_0 < e_1 < \cdots < e_N$. Suppose for some $k \geq 0$
the subsequence $(e_{k+1}, \ldots, e_{k+m})$ is equal to 
$(1, 2 \ldots, m)$. Then $R^l p_{1*} \gF = 0$ for $l > 0$ and $p_{1*} \gF$ has
a resolution on $X$ of the form 
\begin{eqnarray*}
 & \gG_0 \otimes H^0 \opm(-e_0) \vpil 
& \gG_1 \otimes H^0 \opm(-e_1) \vpil \cdots \\
\vpil & \gG_k \te H^0 \opm (-e_k) \vpil & \gG_{k+m+1} \te H^m \opm(-e_{k+m+1}) \vpil
\cdots \\
\vpil & \gG_N \te H^m \opm(-e_N). & 
\end{eqnarray*}
\end{proposition}

The proof of this is quite short and uses the hypercohomology spectral sequence.

\begin{example}
Let $Y$ be the complete intersection of $m$ forms of type $(1,1)$ on 
$\proj{a} \times \proj{b}$.
Assume $m \geq b+d$ where $d$ is a non-negative integer. 
Let $\gF$ be the twisted structure sheaf $\gO_Y(0,d)$.
The resolution of $\gO_Y(0,d)$ is 
\begin{align*}
 & \gO(0,d)^{\alpha_0} \vpil   \gO(-1, d-1)^{\alpha_1} \vpil 
& \cdots &  \vpil &  \gO(-d,0)^{\alpha_d}&   \\ 
\vpil &  \gO(-d-1, -1)^{\alpha_{d+1}} \vpil & \cdots &   
\vpil   & \gO(-d-b, -b)^{\alpha_{d+b}}&  \\
\vpil & \gO(-d-b-1, -b-1)^{\alpha_{d+b+1}} \vpil
& \cdots &   \vpil &  \gO(-m, d-m)^{\alpha_m}.&  
\end{align*}
The first coordinate twist is the one we are interested in. If we push
this complex forward to $\proj{a}$, the above Proposition
\ref{ExiProProj}, shows that $p_* \gO_Y(0,d)$ has
a resolution
\begin{align*} 
\gO(0)^{\alpha^\prime_0} & \vpil  \gO(-1)^{\alpha^\prime_1} 
\vpil \cdots \vpil  \gO(-d)^{\alpha^\prime_d} \\
 & \vpil \gO(-d-b-1)^{\alpha^\prime_{d+b+1}}
\vpil \cdots \vpil \gO(-m)^{\alpha^\prime_{m}}. 
\end{align*}
We see that we adjusted the second coordinate twist so that we got 
a collapse in the first coordinate twist resulting in a gap from 
$d$ to $d+b+1$. We have a complex which is pure but no longer linear.
\end{example}

For the general construction, suppose we want a pure resolution of
a sheaf on $\proj{n}$ by sums of line bundles,
\begin{equation} \label{ExiLigRenO}
\opn(-d_0)^{\alpha_0} \vpil \opn(-d_1)^{\alpha_1} \vpil \cdots 
\vpil \opn(-d_n)^{\alpha_n}.
\end{equation}

  We see in the example above that we get a leap 
in twist $\gO(-d) \vpil \gO(-d-b-1)$ by pushing down omitting a factor
$\proj{b}$. The leaps in the pure resolution we want are $d_i - d_{i-1}$, 
so we consider a projective space $\proj{m_i}$ where $m_i = d_i - d_{i-1}-1 $. 
On the product $\proj{n} \times \proj{m_1} \times \cdots \times \proj{m_n}$
we let $Y$ be the complete intersection of $M = d_n$ forms of type 
$(1,1, \ldots, 1)$, and let $\gF$ be $\gO_Y(0, d_0, d_1, \ldots, d_{n-1})$.
Its resolution is 
\begin{align*}
  & \gO(0,d_0, \ldots, d_{n-1}) \vpil \gO(-1,d_0-1, &\ldots&, d_{n-1}-1 ) 
\vpil \cdots     \\
  \vpil &  \gO(-d_{n-1}, d_0 -d_{n-1}, \ldots, 0)^{\alpha^\prime} \vpil \cdots 
   & \vpil &     
\gO(-d_{n-1}-m_n, d_0 - d_{n-1}-m_n, \ldots, -m_n)^{\alpha^{\prime\prime}}  \\ 
   &  &  \vpil &   
\gO(-d_n, \ldots, -m_n - 1)^{\alpha^{\prime\prime\prime}} 
\end{align*}  
Note that when the coordinate twist corresponding to $\proj{m_i}$
varies through $0, -1, \ldots$  $-m_i, -m_i - 1$ (displayed when
$i = n$ in the second and third line above), the
first coordinate twist varies through
\begin{eqnarray*}
 -d_{i-1}, -d_{i-1} - 1, \ldots, & -d_{i-1} - m_i, & -d_{i-1} - m_i - 1 \\
& (= -d_i + 1), & (= -d_i).
\end{eqnarray*}

Hence after the projection omitting $\proj{m_i}$, only the first
twist $-d_{i-1}$ and the last $-d_i$ survive in the first coordinate.
After all the projections we get a pure resolution
consisting of sums of line bundles (\ref{ExiLigRenO}) 
on $\proj{n}$.
Taking global sections of all twists of this complex, we get a complex
\[ S(-d_0)^{\alpha_0} \vpil \cdots \vpil S(-d_n)^{\alpha_n}. \]
That this is a resolution
follows by the Acyclicity Lemma \cite[20.11]{EFW} or
may be verified by breaking (\ref{ExiLigRenO})
into short exact sequences
\[ 0 \vpil \gK_{i-1} \vpil \gO(-d_i)^{\alpha_i} \vpil \gK_i \vpil 0. \]
By descending induction on $i$ starting from $i = n$ one easily checks
that there are exact sequences of graded modules
\[ 0 \vpil \Gamma_* \gK_{i-1} \vpil S(-d_i)^{\alpha_i} 
\vpil \Gamma_*\gK_i \vpil 0. \]

\subsubsection{Generalizations}
In \cite{BEKS2} this method of collapsing part of the complex
by suitable projections is generalized considerably. 
They construct wide classes of multilinear complexes
from tensors $\phi$ in $R^a \te R^{b_1} \te \cdots \te R^{b_n}$
where $R$ is a commutative ring, and weights $w$ in $\hele^{n+1}$  
(the twist $(0,d_0, d_1, \ldots, d_{n-1})$ above is such a weight).
In a generic setting these are resolutions generalizing many known
complexes, like for instance Eagon-Northcott and Buchsbaum-Rim
complexes which arise from a $2$-tensor (a matrix).

In particular, Theorem 1.9 of \cite{BEKS2} provides infinitely many
new families of pure resolutions of type $\bfd$ for any degree sequence $\bfd$.
The essential idea is that given a degree sequence, say $(0,4,7)$, the
integers in the complement 
\[ \cdots, -3, -2, -1, 1, 2, 3, 5, 6, 8, 9, 10, \cdots \]
may be partitioned in many ways into sequences of successive integers, and by
cleverly adjusting the construction above, all twists in such sequences
may be collapsed.

 The paper also give explicit constructions of the differentials
of the complexes, and in particular
of those in the resolutions constructed above by Eisenbud and Schreyer.

\subsection{Pure resolutions constructed from generic matrices}
\label{SubsekExiDeg}
We now describe the second construction of pure resolutions in \cite{EFW},
which also requires that char $\kr = 0$. It gives a comprehensive
generalization of the Eagon-Northcott and Buchsbaum-Rim complexes
in a generic setting.
\subsubsection{Resolutions of length two.}
For a map $R^{r-1} \mto{\phi} R^r$ of free modules over a commutative
ring $R$, let $m_i$ be the minor we get by deleting row $i$ in the map
$\phi_i$. The well-known Hilbert-Burch theorem says that if the ideal
$I = (m_1, \ldots, m_r)$ has depth $\geq 2$ then there is a resolution
of $R/I$
\[ R \xleftarrow{ [m_1, -m_2,  \ldots, (-1)^r m_r]}
R^r \vmto{\phi} R^{r-1}. \]
We get a generic situation if we let $F$ and $G$ be vector spaces
with bases $f_1, \ldots, f_{r-1}$ and $g_1, \ldots, g_r$ respectively, and 
set $S = \Symm (G^* \te F)$. Then $G^* \te F$ has basis $e_{ij} = g_i^* \te f_j$,
where the $g_i^*$ are a dual basis for $G^*$. We have a generic map
\[ G \tek S \overset{ [e_{ij}]}{\vpil} F \tek S(-1) . \]
The Hilbert-Burch theorem gives a resolution
\[ S \xleftarrow{[m_1, -m_2,  \ldots, (-1)^{r-1} m_r]}
 G \tek S(-r+1) \xleftarrow{ [e_{ij}]} F \tek S(-r). \]

The construction of \cite[Section 4]{EFW} generalizes this to pure
resolutions of type $(0,s,r)$ for all $0 < s < r$. 
This is a resolution
\begin{equation*}
\wedge^{r-1}F \tek \wedge^s F^* \tek S \longleftarrow   
\wedge^{r-1} F \tek \wedge^s G^* \tek S(-s) \longleftarrow
\wedge^{r-1} F \tek \wedge^{r-s} F \tek \wedge^r G^* \tek S(-r).
\end{equation*}
Note that the right map here identifies as the natural map
\[ \wedge^{r-1}F \tek \wedge^r G^* \tek \wedge^{r-s} G \tek S(-s)
\longleftarrow \wedge^{r-1} F \tek \wedge^r G^* \tek
\wedge^{r-s}F \tek S(-r). 
\]

Also note that the ranks of the modules here are different from the 
ranks of the modules in the equivariant case. For instance the rank
of the middle term in this construction is $\binom{r}{s}$ while
in the equivariant construction this rank is simply $r$. 

\begin{example}
When $r = 5$ we derive in the case $s=4$ the Hilbert-Burch complex
\[  S  \xleftarrow{[m_1, -m_2, m_3, -m_4, m_5]}
 S(-4)^5 \xleftarrow{ [e_{ij}]} S(-5)^4. \]

When $s= 3$ we get a complex
\[ S^{\binom{4}{3}} \vmto{\alpha} S(-3)^{\binom{5}{3} = \binom{5}{2}}
\vmto{\beta} S(-5)^{\binom{4}{2}}. \]
There is a natural basis of the first term $ S^{\binom{4}{3}}$ consisting
of three-sets of basis elements $\{ f_{i_1}^*, f_{i_2}^*, f_{i_3}^* \}$ of $F^*$ 
and a basis for  
$S(-3)^{\binom{5}{3}}$ consisting of three-sets of basis elements  
$\{ g_{j_1}^*, g_{j_2}^*, g_{j_3}^* \}$ of $G^*$. The entries of $\alpha$ are the 
$3 \times 3$ -minors corresponding to the columns $i_1, i_2, i_3$ and
rows $j_1, j_2, j_3$ of the $5 \times 4$ matrix $[e_{ij}]$. 
Similarly the entries of $\beta$
consists of $2 \times 2$-minors of the matrix $[e_{ij}]$.

When $s = 2$ we get a complex
\[ S^{\binom{4}{2}} \vmto{\alpha} S(-2)^{\binom{5}{2} = \binom{5}{3}}
\vmto{\beta} S(-5)^{\binom{4}{3}}, \]
and when $s = 1$ we get a complex
\[ S^{\binom{4}{1}} \vmto{\alpha} S(-1)^{\binom{5}{1} = \binom{5}{4}}
\vmto{\beta} S(-5)^{\binom{4}{4}}. \]

\end{example}

\subsubsection{Resolutions of length three and longer.}
To construct resolutions of length three one must start with a 
vector space $F$ of rank $r-2$ and a vector space $G$ of rank $r$. 
The resolution of pure type $(0,a,a+b,a+b+c = r)$ has the following form
\begin{eqnarray*}
S_{2^{c-1},1^{b-1}} F \tek S & \vmto{\alpha_1} & 
\wedge^{r-2} F \tek \wedge^{c-1} F
\tek \wedge^a G^* \tek S(-a) \\
& \vmto{\alpha_2} & \wedge^{r-2} F \tek \wedge^{b+c-1} F \tek
\wedge^{a+b} G^* \tek S(-a-b) \\
& \vmto{\alpha_3} & \wedge^{r-2} F \tek S_{2^c,1^{b-1}} F \tek \wedge^r G^*
\tek S(-r).
\end{eqnarray*}
When $b = c = 1$ this is the Eagon-Northcott complex
associated to a generic map.  When $a = c = 1$ 
it is the Buchsbaum-Rim complex, and when $a = b = 1$ we get the
third complex occurring naturally in this family as given in 
\cite[Appendix A.3]{Ei}.

\medskip
In general for a degree sequence $\bfd = (d_0, \ldots, d_c)$, 
denote $e_i = d_i - d_{i-1}$. 
One chooses $G$ of rank $r = \sum_1^c e_i$ and $F$ of rank $r-c+1$. 
Let $\gamma(\bfe, i)$ be the partition
\[ ((c-1)^{e_c - 1}, (c-2)^{e_{c-1} - 1}, \ldots, i^{e_{i+1} - 1}, 
i^{e_i}, (i-1)^{e_{i-1} - 1}, \ldots, 1^{e_1 - 1}). \]
This is the dual of the partition $\alpha(\bfe, i)$ defined in 
the equivariant case.
The terms in our complex will be
\[ H(\bfd, i) = S_{\gamma(\bfe, i)} F \tek \wedge^{d_i} G^* 
\tek S(-d_i). \]
The differentials $H(\bfd, i) \mto{\gamma_i} H(\bfd, i-1)$ 
in the complex are given by 
\begin{gather*}
 S_{\gamma(\bfe, i)} F \tek  \wedge^{d_i} G^* \tek S(-d_i) \\
 \downarrow \\
S_{\gamma(\bfe, i-1)} F \tek \wedge^{d_i - d_{i-1}} F 
\tek  \wedge^{d_i - d_{i-1}} G^* \tek \wedge^{d_{i - 1}} G^*
\tek S(-d_i) \\
 \downarrow \\
S_{\gamma(\bfe, i-1)}F \tek  \wedge^{d_{i-1}} G^* \tek S(-d_{i-1}). 
\end{gather*}

The last map is due to 
$ \wedge^{d_i - d_{i-1}} F 
\tek \wedge^{d_i - d_{i-1}} G^*$
being a summand of $\Symm(F \te G^*)_{d_i - d_{i-1}}$. 

\begin{theorem}[Theorem 0.2, \cite{EFW}]
 The complex $H(\bfd,\cdot)$ is a $GL(F) \times GL(G)$ equivariant 
pure resolution of type $\bfd$. 
\end{theorem}

\section{Cohomology of vector bundles
on projective spaces}
\label{SekBunt}
In their paper \cite{ES}, Eisenbud and Schreyer also achieved
a complete classification of cohomology tables of vector bundles
on projective spaces up to a rational multiple.
This runs fairly analogous to the classification of 
Betti diagrams of Cohen-Macaulay modules up to rational multiple.
First we introduce cohomology tables of coherent sheaves and vector 
bundles, and notation related to these.

\subsection{Cohomology tables}
For a coherent sheaf $\gF$ on the projective space $\proj{m}$ 
our interest shall be 
the cohomological dimensions
\[ \gamma_{i,d}(\gF) = \dim_\kr H^i \gF(d). \] 
The indexed set $(\gamma_{i,d})_{{i = 0, \ldots,m},{d \in \hele}}$ is 
the {\it cohomology table} of $\gF$, which lives in the vector
space  ${\mathbb T} = {\mathbb D}^* = \Pi_{d \in \hele} \rat^{m+1}$ with the 
$\gamma_{i,d}$
as coordinate functions. An element in this vector space will be called
a {\it table}.

 We shall normally display
a table as follows. 
\begin{center}
\begin{tabular}{ c c c c c | c}
$\cdots$ & $\gamma_{n, -n-1}$  & $\gamma_{n, -n}$ & $\gamma_{n, -n+1}$ & 
$\cdots$ & $n$ \\ 
 & \vdots  & \vdots & \vdots & \\
$\cdots$ &  $\gamma_{1, -2}$  & $\gamma_{1, -1}$ & $\gamma_{1, 0}$ & 
$\cdots$ & $1$ \\
$\cdots$ &  $\gamma_{0, -1}$  & $\gamma_{0, 0}$ & $\gamma_{0, 1}$ & $\cdots$ & $0$ 
\\
\hline
$\cdots$  & $-1$ & $0$ & $1$ & $\cdots$ & $d \backslash i$
\end{tabular}.
\end{center}

Compared to the natural way of displaying $\gamma_{i,d}$ in 
row $i$ and column $d$, we have shifted row $i$ to the right $i$ steps.
With the above way of displaying the cohomology table, the columns
correspond to the terms in the Tate resolution 
(see  Subsection \ref{SubsekFurPro}) of the coherent sheaf $\gF$.
We write $H^i_* \gF = \oplus_{n \in \hele} H^i \gF(n)$. This is an 
$S$-module, the $i$'th cohomology module of $\gF$.

\begin{example}

The cohomology table of the ideal sheaf of two points in 
$\proj{2}$ is 

\begin{center}
\begin{tabular}{c c c c c c c c c c} 
$\cdots$ & 6 & 3 & 1 & $\lfloor$0 & 0 & 0 & 0 & $\cdots$ \\ 
$\cdots$ & 2 & 2 & 2 & 2 & 1 & $\lfloor$0 & 0 & $\cdots$ \\
$\cdots$ & 0 & 0 & 0 & 0 & 1 & 3 & 8 & $\cdots$ \\
\hline
 $\cdots$ & -3 & -2 & -1 & 0 & 1 & 2 & 3 & $\cdots$
\end{tabular}
\end{center}

In this table there are in the two upper rows two distinguished
corners with $0$, indicated with a $\lfloor$, 
such that the quadrant determined by it only 
consists of zeroes. The $0$ in  the $H^1$ -row is in the column
labelled by $2$ so it is in cohomological degree $z_1 = 2- 1$.
The $0$ in the $H^2$-row is in the column labelled by $0$
so its degree is $z_2 = 0 - 2$. The sequence 
$z_1, z_2$ is called the root sequence of the 
cohomology table.
\end{example}

Recall that the classical Castelnuovo-Mumford regularity of a coherent
sheaf $\gF$ is defined by
\[ r = \inf \{ m | H^i \gF (m-i) = 0 \mbox{ for } i \geq 1 \}. \]

\begin{definition} For $p \geq 1$ the $p$-{\it regularity}  of a coherent
sheaf is defined to be
\[ r_p = \inf \{ d | H^i \gF(m-i) = 0 \mbox{ for } i \geq p, m \geq d \}. \] 
(Is is not difficult to show that the numbers  $r_1$ and $r$ are the same.)  
The {\it root sequence} of $\gF$ is $z_p = r_p - p$ for $p \geq 1$.
\end{definition}
(Eisenbud and Schreyer call in \cite{ES3} the $z_p$ the regularity sequence, 
but by private communication from Schreyer the notions of root sequence
and regularity sequence were mixed up in that paper.)

\begin{example} \label{BuntEksCoh}
Let $\gE$ be the vector bundle on ${\mathbb P}^3$
which is the cohomology of the complex
\[ \op{3}  \overset{[x_0, x_1, x_2^2, x_3^2]}{\vpil}
\op{3}(-1)^2 \oplus \op{3}(-2)^2 
\overset{ [-x_2^2, -x_3^2, x_0, x_1]^t} {\vpil}
\op{3}(-3). 
\]
The cohomology table of this is 
\begin{center}
\begin{tabular}{c c c c c c c c c| c}
$\cdots$ & 21 & 7  & 1 & 0 &0 &0 & 0 &  $\cdots$ & 3\\ 
$\cdots$ & 0$\rceil$ & 1 & 2& 1 &0 &0 &0 &  $\cdots$ & 2\\ 
$\cdots$ & 0 & 0 & 0$\rceil$ & 1 & 2 & 1 & 0 &  $\cdots$ & 1\\
$\cdots$ & 0 & 0 & 0 & 0$\rceil$ & 1 & 7 & 21 &  $\cdots$ & 0\\
\hline
$\cdots$ &  -2 & -1 & 0 & 1 & 2 & 3 & 4 & $\cdots$ & $d \backslash i$
\end{tabular}.
\end{center}

\end{example}

For a vector bundle $\gE$ all the intermediate cohomology modules
$H^i_* \gE$ have finite length for $i = 1, \ldots, m-1$. 
Also $H^0 \gE(d)$ vanishes for $d \ll 0$. Hence in this
case it is for rows $0,1,\ldots, m-1$ also meaningful to speak of 
the corners with $0$, extending
downwards and to the left, indicated by $\rceil$ in the 
diagram above.

\subsection{The fan of cohomology tables of vector bundles}
We want to consider vector bundles whose cohomology tables live in 
a finite dimensional subspace of ${\mathbb T}$. 
Let $\zdec{m}$ be the set of strictly decreasing integer sequences
$(a_1, \ldots, a_m)$. Such sequences are called {\it root sequences}. 
This is a partially ordered set  with $\bfa \leq \bfb$ if  
$a_i \leq b_i$ for $i = 1,\ldots, m$. The interval 
$[\bfa, \bfb]_{root}$ is the set of all root sequences $\bfz$ such
that $\bfa \leq \bfz \leq \bfb$. 
We will consider vector bundles $\gE$
such that for each $i = 1, \ldots, m$ we have 
$H^i \gE(p) = 0$ for $p \geq b_i$. (This is the same as
the root sequence of $\gE$ being $\leq \bfb$. )

As shown in Example \ref{BuntEksCoh},
for a vector bundle we may also bound below the ranges of the cohomology
modules $H^i_* \gE$ for $i = 0, \ldots, m-1$, and
we assume that $H^i \gE(p) = 0$ for $p \leq a_{i+1}$ for $i = 0, \ldots, m-1$. 
In particular note that $b_i-1$ bounds above the supporting range
of $H^i_* \gE$ and $a_{i+1}+1$ bounds below the supporting range of 
$H^{i}_* \gE$. If $\gE$ has supernatural cohomology, the conditions
means that its root sequence is in the interval $[\bfa, \bfb]_{root}$.

\begin{definition}
$\Tab$ is the subspace of ${\mathbb T}$
consisting of all tables
such that 
\begin{itemize}
\item $\gamma_{i,d} = 0$ for $i = 1, \ldots, m$ and $d \geq b_i$. 
\item $\gamma_{i,d} = 0$ for $i = 0, \ldots, m-1$ and $d \leq a_{i+1}$. 
\item The alternating sum $\gamma_{0,d} - \gamma_{1,d} + \cdots +
(-1)^m \gamma_{m,d}$ 
is a polynomial in $d$ of degree $\leq m$ for $d \geq b_1$ and
for $d \leq a_m$.
\end{itemize}
\end{definition}

The space $\Tab$ is a finite dimensional vector space
as is easily verified, since the values of a polynomial of degree $\leq m$
is determined by any of $m+1$ successive values.
The last condition for $\Tab$ is not really canonical. The conditions are just
to
get a simply defined finite-dimensional space containing
the cohomology tables of vector bundles with supernatural cohomology
and root sequences in the interval $[\bfa, \bfb]_{root}$. 
Note that set of all positive rational
multiples of cohomology tables of vector bundles whose tables are in $\Tab$, 
forms a positive cone which we denote by $C(\bfa, \bfb)$.

\medskip

For a root sequence $\bfz : z_1 > z_2 > \cdots > z_m$ we associate a table
$\gamma^\bfz$ given by 
\[ \gamma^\bfz_{i,d} = \begin{cases} \frac{1}{m!} \Pi_{i = 1}^m |d- z_i| 
& z_i > d > z_{i+1} \\
0 & \mbox{otherwise}. \end{cases}
\]
This is the {\it supernatural table} associated to this root sequence.

\begin{lemma}
If $Z : \bfz^1 > \bfz^2 > \cdots > \bfz^r$ is a chain of root sequences, then
$\gamma^{\bfz^1}, \gamma^{\bfz^2}, \cdots ,  \gamma^{\bfz^r}$ 
are linearly independent.
\end{lemma}

Hence these supernatural tables
span a simplicial cone $\sigma(Z)$ in ${\mathbb T}$.

\begin{proposition} The set of simplicial cones $\sigma(Z)$ where
$Z$ ranges over the chains $Z : \bfz^1 < \bfz^2 < \cdots < \bfz^r$
in $[\bfa, \bfb]_{\dec}$, form a simplicial fan in $\Tab$ which we denote as 
$\Gamma(\bfa, \bfb)$. 
\end{proposition}

Here is the analog of Theorem \ref{ResTheGeo}.

\begin{theorem} \label{BuntTheHoved} \begin{itemize}
\item[a.] The realization of the fan $\Gamma(\bfa, \bfb)$ is contained
in the positive cone $C(\bfa, \bfb)$. 
\item[b.] The positive cone $C(\bfa, \bfb)$ is contained in the
realization of the fan $\Gamma(\bfa, \bfb)$. 
\end{itemize}

In conclusion the realization of $\Gamma(\bfa, \bfb)$ and the positive
cone $C(\bfa, \bfb)$ are equal.

\end{theorem}

Part a. is a consequence of the existence of vector bundles with 
supernatural cohomology which is treated in Subsections \ref{SubsekExiEkvisup}
and \ref{SubsekExiChfreesup}. The proof of part b. is analogous
to the proof of Theorem \ref{ResTheGeo} part b., which we
developed in Section \ref{SekFacet}. We outline this in the next
subsection and the essential part is again to find the facet
equations of $\Gamma(\bfa, \bfb)$.  

\subsection{Facet equations}

\begin{example} \label{BuntEksChains}
Let $\bfa = (0,-4, -5)$ and $\bfb = (0,-2,-4)$. 
Considering the interval $[\bfa, \bfb]_{root}$ as a partially ordered
set, its Hasse diagram is:

\vskip 4mm
\hskip 3.5cm
\scalebox{1} 
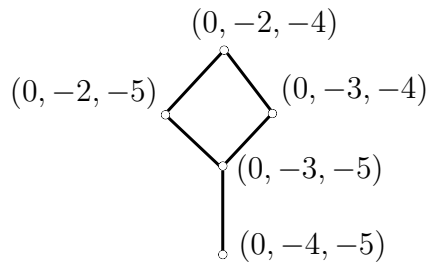
{
\begin{pspicture}(0,-1.6875)(5.5471873,1.6909375)
\psdots[dotsize=0.12,fillstyle=solid,dotstyle=o](2.781875,1.1290625)
\psline[linewidth=0.04cm](2.741875,1.0890625)(2.001875,0.2890625)
\psline[linewidth=0.04cm](2.821875,1.0890625)(3.401875,0.3290625)
\psdots[dotsize=0.12,fillstyle=solid,dotstyle=o](3.421875,0.2890625)
\psdots[dotsize=0.12,fillstyle=solid,dotstyle=o](2.001875,0.2690625)
\psline[linewidth=0.04cm](2.041875,0.2290625)(2.721875,-0.4109375)
\psdots[dotsize=0.12,fillstyle=solid,dotstyle=o](2.761875,-0.4109375)
\psline[linewidth=0.04cm](3.381875,0.2490625)(2.801875,-0.3709375)
\psline[linewidth=0.04cm](2.761875,-0.4709375)(2.761875,-1.5309376)
\psdots[dotsize=0.12,fillstyle=solid,dotstyle=o](2.761875,-1.5909375)
\usefont{T1}{ptm}{m}{n}
\rput(3.3154688,1.4790626){$(0,-2,-4)$}
\usefont{T1}{ptm}{m}{n}
\rput(4.5154686,0.5990625){$(0,-3,-4)$}
\usefont{T1}{ptm}{m}{n}
\rput(0.91546875,0.5590625){$(0,-2,-5)$}
\usefont{T1}{ptm}{m}{n}
\rput(3.9154687,-0.4609375){$(0,-3,-5)$}
\usefont{T1}{ptm}{m}{n}
\rput(3.9554687,-1.4809375){$(0,-4,-5)$}
\end{pspicture} 
}

There are two maximal chains in this diagram
\begin{eqnarray*}
Z & : & (0,-4,-5) < (0,-3,-5) < (0,-3,-4) < (0,-2,-4) \\
Y & : & (0,-4,-5) < (0,-3,-5) < (0,-2,-5) < (0,-2,-4)
\end{eqnarray*}
so the realization of the Boij-S\"oderberg fan consists of
the union of two simplicial cones of dimension four. 
Cutting it with a hyperplane, we get two tetrahedra.
(The vertices are labelled by the pure diagrams on its rays.)
\vskip 4mm
\hskip 3cm
\scalebox{1} 
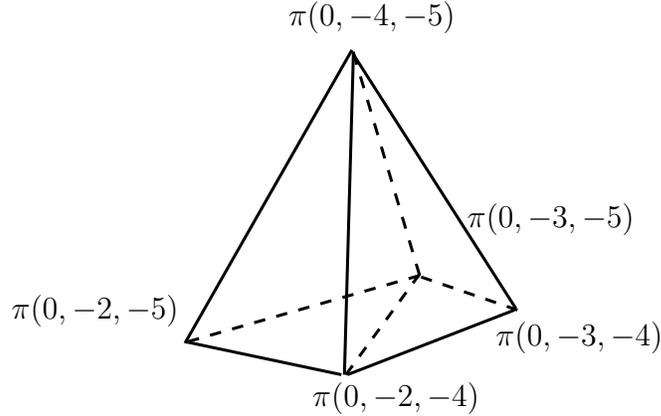
{
\begin{pspicture}(0,-2.695625)(8.287188,2.695625)
\psline[linewidth=0.04cm](3.981875,2.05375)(1.761875,-1.84625)
\psline[linewidth=0.04cm](4.001875,2.03375)(3.881875,-2.26625)
\psline[linewidth=0.04cm](4.001875,1.99375)(6.161875,-1.38625)
\psline[linewidth=0.04cm,linestyle=dashed,dash=0.16cm 0.16cm](4.881875,-0.92625)(3.901875,-2.26625)
\psline[linewidth=0.04cm,linestyle=dashed,dash=0.16cm 0.16cm](4.901875,-0.96625)(6.181875,-1.40625)
\psline[linewidth=0.04cm,linestyle=dashed,dash=0.16cm 0.16cm](4.881875,-0.94625)(4.001875,2.01375)
\psline[linewidth=0.04cm,linestyle=dashed,dash=0.16cm 0.16cm](1.801875,-1.82625)(4.881875,-0.94625)
\psline[linewidth=0.04cm](1.761875,-1.82625)(3.841875,-2.26625)
\psline[linewidth=0.04cm](3.921875,-2.26625)(6.141875,-1.40625)
\usefont{T1}{ptm}{m}{n}
\rput(4.2454686,2.48375){$\pi(0,-4,-5)$}
\usefont{T1}{ptm}{m}{n}
\rput(7.005469,-1.77625){$\pi(0,-3,-4)$}
\usefont{T1}{ptm}{m}{n}
\rput(4.565469,-2.57625){$\pi(0,-2,-4)$}
\usefont{T1}{ptm}{m}{n}
\rput(0.5654687,-1.41625){$\pi(0,-2,-5)$}
\usefont{T1}{ptm}{m}{n}
\rput(6.6254687,-0.19625){$\pi(0,-3,-5)$}
\end{pspicture} 
}

There is one interior facet of the fan, while all other facets are 
exterior. The exterior facets are of three types. We give an example
of each case by giving the chain.

\begin{itemize} 
\item[1.] $Z \backslash \{ (0,-2,-4)\}$. Here we omit the maximal 
element $\bfb$. Clearly this can only be completed to a maximal chain
in one way so this gives an exterior facet.
The
nonzero values of the table $\gamma^{(0,-2, -4)}$ is 

\begin{center}
\begin{tabular}{ c c c   c c c   c c c}
$\cdots$  & * & * & *  & . & . & . & . &  $\cdots$ \\ 
$\cdots$  & . & . & .  & * & . & . & . &  $\cdots$ \\ 
$\cdots$  & . & . & .  & . & * & . & . &  $\cdots$ \\ 
$\cdots$  & . & . & .  & . & . & * & * & $\cdots$ \\ 
\hline
$\cdots$  & -4 & -3&-2 & -1 & 0 & 1 & 2  & $\cdots$ \\ 
 \end{tabular}
\end{center}
The second coordinate is changing from $\bfb$ to its predecessor. 
Hence the facet equation is $\gamma_{2,-3} = 0$, since 
$\gamma_{2,-3} $ is nonzero on $\gamma^{(0,-2,-4)}$ but vanishes on 
the other elements in $Z$.

\item[2.] $Z \backslash \{ (0,-3,-5) \}$. This chain contains 
$(0,-3,-4)$ and $(0,-4,-5)$. Clearly the only way to complete this to a
maximal chain is by including $(0,-3,-5)$, so this gives an exterior facet.
The tables associated to these root sequences has nonzero positions
as follows

\begin{center}
\begin{tabular}{ c c c   c c c   c c c}
$\cdots$  & * & * & +         & . & . & . & . &  $\cdots$ \\ 
$\cdots$  & . & . & $\sim$          & . & . & . & . &  $\cdots$ \\ 
$\cdots$  & . & . & -         & * & * & . & . &  $\cdots$ \\ 
$\cdots$  & . & . & .         & . & . & * & * & $\cdots$ \\ 
\hline
$\cdots$  & -4 & -3&-2 & -1 & 0 & 1 & 2  & $\cdots$ \\ 
 \end{tabular}
\end{center}
In column $-2$ each of $\gamma^{(0,-3,-4)}, \gamma^{(0, -3, -5)}$ and 
$\gamma^{(0, -4,-5)}$ has only one nonzero value, indicated by
a $+, \sim$ and $-$ respectively. We see that $\gamma_{2, -4}$ 
is nonzero on $\gamma^{(0,-3,-5)}$ but vanishes on the other
elements in the chain, giving the facet equation $\gamma_{2, -4} = 0$.

\item[3.] $Y \backslash \{(0,-3,-5)\}$. This chain contains $(0,-2,-5)$ and 
$(0,-4,-5)$. Clearly the only way to complete this to a
maximal chain is by including $(0,-3,-5)$, so this gives an exterior facet.
The nonzero 
cohomology groups of $\gamma^{(0,-2,-5)}$ are indicated by $*$ and $+$ 
in the following diagram, those of $\gamma^{(0,-4,-5)}$ are indicated
by $*$ and $-$, while those of the element omitted, $\gamma^{(0,-3, -5)}$, are
indicated by $*$'s, the first $+$ and the second $-$. 
\begin{center}
\begin{tabular}{ c c c  c  c c c   c c c}
$\cdots$  & * & * & *  & . & . & . & . & . & $\cdots$ \\ 
$\cdots$  & . & . & .  & + & + & . & . & . & $\cdots$ \\ 
$\cdots$  & . & . & .  & - & - & * & . & . & $\cdots$ \\ 
$\cdots$  & . & . & .  & . & . & . & * & * & $\cdots$ \\ 
 \end{tabular}
\end{center}
 
The diagram is divided into two parts. The upper part consists of all 
positions above the $*$ and $-$ positions, and the lower part below
the $*$ and $+$ positions. There will be an upper and a lower facet 
equation. 
Working it out in a way analogous to Example \ref{FacExUp}, 
the lower facet equation is given by the following
table
\begin{center}
\begin{tabular}{ c c c  c  c c c   c c c}
$\cdots$  & $0^*$ & $0^*$ & $0^*$ & 0 & 0 & 0 & 0 & 0 & $\cdots$ \\ 
$\cdots$  & 0 & 0 & 4 & $0^+$ & $0^+$ & 0 & 0 & 0 & $\cdots$ \\ 
$\cdots$  & 0 & -4 & 15 & -20  & 10 & $0^*$ & 0 & 0 & $\cdots$ \\ 
$\cdots$  & 4 & -15 & 20 &  -10  & 0 & 1  & $0^*$ & $0^*$ & $\cdots$ \\ 
 \end{tabular}
\end{center}
\end{itemize}

The meaning of the numbers turns out to be as follows. Taking the negative
of the union of the degree sequences of $z^+, z$ and $z^-$ we get
$\bfd = (0, 2, 3, 4, 5)$. The pure resolution of this type has
exactly the absolute values of the 
nonzero numbers in the bottom row as Betti numbers:
\[ S \vpil  S(-2)^{10} \vpil S(-3)^{20} \vpil S(-4)^{15} \vpil S (-5)^{4}.
\]
\end{example}

\begin{proposition}
Let $Z$ be a maximal chain in $[\bfa, \bfb]_\dec$ and $z \in Z$. 
Then $\sigma(Z \backslash \{z \})$ is an exterior facet 
of $\Gamma(\bfa, \bfb)$ if either
of the following holds.
\begin{itemize}
\item[1.] $z$ is either $\bfa$ or $\bfb$. The facet equation is 
$\gamma_{i,d} = 0$
for appropriate $i$ and $d$.
\item[2.] The root sequences of $z^-$ and $z^+$ immediately
before and after $z$ in $Z$ differ in exactly one position. So for 
some $r$ we have 
\[ z^- = (\cdots, -(r+1), \cdots), z = (\cdots, -r, \cdots), 
z^+ = (\cdots, -(r-1), \cdots). \]
\item[3.]  The root sequences of $z^-$ and $z^+$ immediately
before and after $z$ in $Z$ differ in two consecutive positions such that
for some $r$ we have 
\begin{eqnarray*}  z^- = (\cdots, -r,-(r+1), \cdots), 
&  z = & (\cdots, -(r-1), -(r+1), \cdots), \\
&  z^+ = & (\cdots, -(r-1), -r, \cdots).
\end{eqnarray*}
Letting $i$ be the position of $-(r-1)$, 
the facet equation is $\gamma_{i,-r} = 0$. 
\end{itemize}
\end{proposition}

For facets of type 2 the description of the facet
equations are as follows.

\begin{theorem} \label{BuntTheU}
Let $Z$ be a chain giving an exterior facet of type 2, and let 
$z^-, z$ and $z^+$ be successive elements in this chain which differ
only in the $i$'th position.
Let $f$ be the degree sequence which is the union of $z^+, z$ and $z^-$
and let $\Fd$ be a pure resolution corresponding to the degree sequence
$f$. The facet equation of this exterior facet is then 
\[  \langle \beta(\Fd), \gamma \rangle_{e, i} = 0 \]
where $e = -z_i -1$. 
\end{theorem}

We may now prove Theorem \ref{BuntTheHoved} b.

\begin{proof} 
Consider a facet of type 2 of the fan
$\Gamma(\bfa, \bfb)$ associated to the root sequence $z$
and position $i$. The upper
equation is $\langle \beta(\Fd), - \rangle_{e,i} = 0$
where $e = -z_i - 1$ and $\Fd$ is given in Theorem \ref{BuntTheU}. 
For facets of type 1 or 3 the hyperplane equations 
are $\gamma_{i,d}  = 0$ for suitable $i,d$. 

Each exterior facet determines a non-negative half plane $H^+$. Since 
the forms above are non-negative
on all cohomology tables $\gamma(\gE)$ in $\Tab$ by Theorem
\ref{FacetThePairing}, the cone $C(\bfa, \bfb)$
is contained in the intersection of all the half planes $H^+$
which again is contained in the 
fan $\Gamma(\bfa, \bfb)$. 
\end{proof}

\section{Extensions to non-Cohen-Macaulay modules 
and to coherent sheaves}
\label{SekExt}
We have in Sections 1 and 2 considered Betti diagrams of 
Cohen-Macaulay modules over $S = \kr[x_1, \ldots, x_n]$ of a 
given codimension.
Shortly after Eisenbud and Schreyer proved the Boij-S\"oderberg conjectures,
Boij and S\"oderberg, \cite{BS2},  extended the theorems to the case of 
arbitrary (finitely generated and graded) 
modules over this polynomial ring. The description here is just
as complete as in the Cohen-Macaulay case.

In \cite{ES3} Eisenbud and Schreyer extended the decomposition 
algorithm for vector bundles to a decomposition algorithm for 
coherent sheaves. This cannot however be seen as a final achievement
since it does not give a way to determine if a table is the cohomology
table, up to rational multiple, of a coherent sheaf on a projective space.

\subsection{Betti diagrams of graded modules in general}
 \label{SubsekFurNC}
 The modifications
needed to extend the  Boij-S\"oderberg conjectures
(theorems actually) to graded modules in general are not great. 
Let $\hele^{\leq n+1}_{\inc}$ be the set of increasing sequences
of integers $\bfd = (d_0, \ldots, d_s)$ with $s \leq n$
and consider a partial order on this by letting
\[ (d_0, \ldots, d_s) \geq (e_0, \ldots, e_t) \]
if $s \leq t$ and $d_i \geq e_i$ when $i$ ranges from $0, \ldots, s$. 
Note that if we identify the sequence $\bfd$ with 
the sequence $(d_0, \ldots, d_n)$ where $d_{s+1}, \ldots, d_n$ are all
equal to $ +\infty$,
then this is completely natural.

Associated to  $\bfd$, 
we have a pure diagram $\pi(\bfd)$ by Subsection \ref{ResSubsecPure},
such that any Cohen-Macaulay module of codimension $s$ with 
pure resolution of type $\bfd$, will have a Betti diagram which is
a multiple of $\pi(\bfd)$. Boij-S\"oderberg
prove the following variation of Theorem \ref{ResTheBS2}
for an arbitrary module.

\begin{theorem}
Let $\beta(M)$ be the Betti diagram of a graded $S$-module $M$.
Then there exists positive rational numbers $c_i$
and a chain of sequences 
$\bfd^1 < \bfd^2 < \cdots < \bfd^p$ 
in $\hele^{\leq n+1}_{\inc}$ such that 
\[ \beta(M) = c_1 \pi(\bfd^1) + \cdots + c_p \pi (\bfd^p). \]
\end{theorem}

The algorithm for this decomposition goes exactly as the algorithm 
in Subsection \ref{ResSubsecAlg}.

\begin{example}
Let $M = \kr[x,y,z] /   (x^2, xy, xz^2)$. This is a module with 
Betti diagram
\begin{equation*}
\begin{matrix} 
0 \\ 1 \\ 2 
\end{matrix}
\left [
\begin{matrix}
1 & 0 & 0 & 0 \\
0 & 2 & 1 & 0 \\
0 & 1 & 2 & 1 
\end{matrix} \right ]
\end{equation*}
which can be decomposed as
\begin{equation*}
\frac{1}{5} \cdot \left [
\begin{matrix} 1 & 0 & 0 & 0 \\ 0 & 5 & 5 & 0 \\ 0 & 0 & 0 & 1
\end{matrix} \right ]
+ \frac{1}{10} \cdot \left[
\begin{matrix} 3 & 0 & 0 & 0 \\ 0 & 10 & 0 & 0 \\ 0 & 0 & 15 & 8
\end{matrix} \right ]
+ \frac{1}{6} \cdot \left [
\begin{matrix} 1 & 0 & 0 & 0 \\ 0 & 0 & 0 & 0 \\ 0 & 4 & 3 & 0
\end{matrix} \right ]
+ \frac{1}{3} \cdot \left [
\begin{matrix} 1 & 0 & 0 & 0 \\ 0 & 0 & 0 & 0 \\ 0 & 1 & 0 & 0
\end{matrix} \right ]
\end{equation*}
\end{example}

\subsubsection{ The cone cut out by the facet equations}
\label{SubsubFurNC} 
Let $\bfa, \bfb$ in $\zinc{n+1}$ be degree sequences of length $(n+1)$.
In the linear space $L^{HK}(\bfa, \bfb)$ we know that the 
positive cone cut out by the functionals $\beta_{ij}$
and $\langle - , \gE \rangle_{e,\tau} $,
where $e$ is an integer, $0 \leq \tau \leq n-1$, and
$\gE$ is a vector bundle on $\proj{n-1}$,
is the positive cone $B(\bfa, \bfb)$.
We may then ask what is the positive cone $B_{eq}(\bfa, \bfb)$
cut out by these functionals
in the space of {\it all} diagrams in the window $\Dab$.

Let
$B_{mod}(\bfa, \bfb)$ be the positive cone
consisting of all rational multiples of Betti diagrams 
of graded modules whose diagram is in the window $\Dab$.
Since the functionals are non-negative on all Betti diagrams
of modules, by Proposition \ref{FacetThePairing},
it is clear that $B_{mod}(\bfa, \bfb) \sus B_{eq}(\bfa, \bfb)$. 
In \cite{BS} they describe the facet equations of 
$B_{mod}(\bfa, \bfb)$. They are limits of facet equations of the type
$\langle - , \gE \rangle_{e,\tau} $ where elements in the root sequence 
of $\gE$ tend to infinity. 
This shows that also 
$B_{mod}(\bfa, \bfb) \supseteq B_{eq}(\bfa, \bfb)$.
Hence the cone $B_{eq}(\bfa, \bfb)$ in $\Dab$ cut out
by the functionals is simply $B_{mod}(\bfa, \bfb)$, 
the positive cone generated by all 
Betti diagrams of graded modules with support in in the window $\Dab$. 

When $c = n$, the exterior facets of type 1 (when removing
a minimal element), 2, and 3 in Proposition 
\ref{FacProFacet} are on unique
exterior facets of the full-dimensional 
cone $B_{mod}(\bfa, \bfb) =  B_{eq}(\bfa, \bfb)$
in $\Dab$.
The unique hyperplane equation (up to scalar) of these latter facets are given
by the $\beta_{ij}$ and the upper equation respectively, testifying to 
the naturality of these choices in Section \ref{SekFacet}.

\subsection{Cohomology of coherent sheaves}
\label{SubsekBuntKnipp}
In contrast to the case of vector bundles
the decomposition algorithm for coherent sheaves on 
projective space is not of a finite number of steps.

In order to extend the algorithm we need to 
define sheaves with supernatural cohomology. Let 
\[ \bfz : z_1 > z_2 > \cdots > z_s \]
be a sequence of integers. It will be convenient to let $z_0 = \infty$
and $z_{s+1} = z_{s+2} = \cdots= - \infty$. A coherent sheaf $\gF$ on 
$\proj{m}$ has supernatural cohomology if:
\begin{itemize}
\item[a.] The Hilbert polynomial is $\chi \gF(d) = \frac{d^0}{s!}
\cdot \Pi_{i = 1}^s (d-z_i)$ for a constant $d^0$ (which must be the
degree of $\gF$). 
\item[b.] For each $d$ let $i$ be such that $z_i > d > z_{i+1}$.
Then 
\[  H^i \gF(d) = \begin{cases}  \frac{d^0}{s!}
\cdot \Pi_{i = 1}^s |d-z_i|, & z_i > d > z_{i+1} \\
 0, & \mbox{otherwise}
\end{cases} \]
\end{itemize}

In particular we see that for each $d$ there is at most one nonvanishing
cohomology group.

The typical example of such a sheaf is a vector bundle with supernatural
cohomology living on a linear subspace $\proj{s} \sus \proj{m}$.
Let $\gamma^\bfz$ be the cohomology table of the sheaf with supernatural
cohomology with root sequence $\bfz$. 

We need to define one more notion derived from a cohomology table
of a coherent sheaf.

\begin{example} Consider the cohomology table :
 \begin{center}
\begin{tabular}{c c c c c c c c c | c}
$\cdots$ & 23 & 11 & 5 & $\lfloor 1$ & 0 & 0 & 0 &  $\cdots$ & 3 \\ 
$\cdots$ & 6 & 5 & 4 & 3 & 2 & $\lfloor 1$ & 0 &  $\cdots$ & 2\\ 
$\cdots$ & 0 & 0 & 1 & 1 & 1 & 1 & 0 &  $\cdots$ & 1\\
$\cdots$ & 0 & 0 & 0 & 0 & 1 & 3 & 8  & $\cdots$ &  0\\
\hline
$\cdots$ &  -2 & -1 & 0 & 1 & 2 & 3 & 4 & $\cdots$ & $d \backslash i$.
\end{tabular}
\end{center}
In rows $3$  and $2$ there are two distinguished corners with nonzero
values, marked with a $\lfloor$ such that in the first quadrant
determined by them, these are the only nonzero values.
In this case the root sequence is $z_1 = 4-1 = 3$, $z_2 = 4-2 = 2$ and 
$z_3 = 2-3 = -1$. We see that there is no corner position in row $1$
because $z_2 = z_1 - 1$. 
\end{example}

\begin{definition} Given a root sequence $z_1 > \cdots > z_s$. 
The position $(i,d) = (i, z_i - 1)$ is a {\it corner
position} if $z_{i+1} < z_i - 1$.
\end{definition}
 
We may verify that 
$\gamma^\bfz$ has nonzero values at each corner position.
Assume $\bfz$ is the root sequence of the cohomology table $\gamma$
of a coherent sheaf.  
Let $\alpha_r, \alpha_{r-1}, 
\ldots, \alpha_0$ be the values of the corner positions of $\gamma$, and
let $a_r, a_{r-1}, \ldots, a_0$ be the values of  the corresponding 
corner positions in $\gamma^\bfz$.

Define 
\[ q_{\bfz} = \min \{ \frac{\alpha_0}{a_0}, \cdots, \frac{\alpha_r}{a_r} \}. \]

Eisenbud and Schreyer \cite{ES3} show the following.

\begin{itemize}
\item The table $\gamma - q_{\bfz} \gamma^\bfz$ has non-negative entries.
\item The root sequence $\bfz^\prime$ of this new table is 
$<$ than the root sequence $\bfz$. 
\end{itemize}

The algorithm of Eisenbud and Schreyer is now to continue this process.
For a table $\gamma$, let $\dim \gamma$ be the largest $i$ such that row 
$i$ is nonzero.

\begin{itemize}
\item[0.] Let $s = \dim \gamma$ and $\gamma_0 = \gamma$. 
\item[1.] $\gamma_1 = \gamma_0 - q_{\bfz_0} \gamma^{\bfz^0}$ 
where $\bfz^0$ is the root sequence of $\gamma_0$. 
\item[2.] $\gamma_2 = \gamma_1 - q_{\bfz_1} \gamma^{\bfz^1}$
where $\bfz^1$ is the root sequence of $\gamma_1$.  \\
$\vdots$
\end{itemize}

In the case of vector bundles in $\Tab$ we are guaranteed that
this process stops at latest when $\bfz^i = \bfa$, and we get
the decomposition derived from the simplicial fan structure
of $\Cab$, Theorem \ref{BuntTheHoved}.
For coherent sheaves
this process gives a strictly decreasing chain of root sequences
\[ \bfz^0 > \bfz^1 > \bfz^2 > \cdots   \]
and may continue an infinite number of steps. Clearly the top value
$z^i_s$ must tend to $- \infty$ as $i$ tends to infinity.
In the end we get a table $\gamma_\infty$ where row $s$ is zero 
so $\dim \gamma_{\infty} < s$. Note that we are not guaranteed that
the entries of $\gamma_\infty$ are rational numbers.

We may repeat this process with $\gamma^\prime = \gamma_\infty$,
which has dimension strictly smaller than that of $\gamma$. 
Eisenbud and Schreyer \cite{ES3} show the following.

\begin{theorem} Let $\gamma(\gF)$ be a cohomology table of
a coherent sheaf $\gF$ on $\proj{m}$. There is a chain of root sequences
$Z$ and positive real numbers $q_{\bfz}$ for $\bfz \in Z$ such that 
\[ \gamma(\gF) = \sum q_{\bfz} \gamma^\bfz. \] 

Both $Z$ and the numbers $q_{\bfz}$ are uniquely determined by these 
conditions. The $q_{\bfz}$ are rational numbers if 
$\dim \gamma^\bfz = \dim \gamma$. 
\end{theorem}

The way $q_{\bfz}$ is defined we are only sure that the corner values of 
$\gamma - q_{\bfz} \gamma^\bfz $ stays non-negative.
The essential ingredient in the proof is to show that not only the corner
values stay non-negative but that every entry in the table stays non-negative.
In order to prove the theorem, Eisenbud
and Schreyer show that certain linear functionals are 
non-negative when applied to the cohomology table of a coherent sheaf.

\section{Further topics}
\label{SekFur}

\subsection{The semigroup of Betti diagrams of modules}
\label{SubsekFurInt}
Theorem \ref{ResTheBS1} gives a complete description of the positive 
rational cone $B(\bfa, \bfb)$ generated by Betti diagrams 
of Cohen-Macaulay modules in the 
window $\Dab$. Of course a more ultimate goal is to describe precisely
what the possible Betti diagrams of modules really are.

This is a much harder problem and the results so far may mostly be described
as families of examples. Investigations into this has been done
mainly by D.Erman in \cite{Er} and by Eisenbud, Erman and Schreyer in 
\cite{EES}. 

Denote by $B_\integ = B(\bfa, \bfb)_\integ$ the semigroup of integer diagrams in 
$B(\bfa, \bfb)$, which we call the semigroup of {\it virtual Betti diagrams},
and let $B_\module = B(\bfa, \bfb)_\module$ be the semigroup of diagrams
in $B(\bfa, \bfb)$ 
which are actual Betti diagrams of modules of codimension $n$.

As a general result Erman shows:
\begin{theorem}[\cite{Er}] The semigroups $B_\integ$ and $B_\module$
are finitely generated.
\end{theorem}

Not every virtual Betti diagram may be an actual Betti diagram of a module.

\begin{example} \label{FurEksEnto}
 The pure diagram $\pi = \pi(0,1,3,4)$ is 

\[ \left [ 
\begin{matrix}
1 & 2 & - & - \\
- & - & 2 & 1
\end{matrix}
\right ].
\]

If this were the Betti diagram of a module, this module would have resolution

\[ S \vmto{\left [ \begin{matrix} l_1, & l_2 \end{matrix} \right ]} 
S(-1)^2 \vpil 
S(-3)^2 \vpil S(-4). 
\]
But this is not possible since writing $S(-1)^2 = Se_1 \oplus S{e_2}$
with $e_i \mapsto l_i$, there would be a syzygy $l_2e_1 - l_1 e_2$ of 
degree $2$. 

However $2 \pi$ is an actual Betti diagram. 
Take a sufficiently general map $S^2 \vmto{d} S(-1)^4$, for instance 
\[ d = \begin{pmatrix} x_1 & x_2 & x_3 & x_4 \\
x_1 & 2x_2 & 3x_3 & 4x_4 
\end{pmatrix}. \]
The resolution of the cokernel of $d$ is then
\[  S^2 \vmto{d} S(-1)^4 \vpil 
S(-3)^4 \vpil S(-4)^2. 
\]
Also, the equivariant resolution $E(1,2,1)$ (recall that $1,2,1$ are
the differences of $0,1,3,4$) is given by 
\[  S^3 \vpil S(-1)^6 \vpil 
S(-3)^6 \vpil S(-4)^3. 
\]
So we see that on the ray determined by $\pi$ the integer
diagrams 
\[   \left [ 
\begin{matrix}
m & 2m & - & - \\
- & - & 2m & m
\end{matrix}
\right ]
\]
are actual Betti diagrams for $m \geq 2$ but not for $m = 1$. 
\end{example}

Recall that if $\bfd$ is a degree sequence, then $\pi(\bfd)$ is the
smallest integer diagram on the ray $t \pi(\bfd), t >0$. The integer
diagrams are then $m \pi(\bfd), m \in \nat$. 
\begin{conjecture}[\cite{EFW}, Conj. 6.1] For every degree sequence 
$\bfd$ there
is an integer $m_0$ such that for $m \geq m_0$ the diagram $m\pi(\bfd)$ 
is the Betti diagram
of a module.
\end{conjecture}

For rays in the positive cone $B(\bfa, \bfb)$ which are not extremal,
things are more refined. 
The following examples are due to Erman \cite{Er}.

\begin{example}
The diagram 
\[ \beta = \left [ 
\begin{matrix}
2 & 3 & 2 & - \\
- & 5 & 7 & 3
\end{matrix}
\right ]
\]
is a virtual Betti diagram in $B_\integ$, and 
\[ \pi = \pi(0,2,3,4) =  \left [ 
\begin{matrix}
1 & - & - & - \\
- & 6 & 8 & 3
\end{matrix}
\right ]
\]
is the Betti diagram of the module $S/(x_1, x_2, x_3)^2$. 
Erman shows that $\beta + m \pi$ is not in $B_\module$ for
any integer $m \geq 0$. 
In particular $B_\integ \backslash B_\module$ may not be finite.
\end{example}

In the following example we let $S = \kr[x_1, \ldots, x_{p+1}]$
where $p$ is a prime. It generalizes 
Example \ref{FurEksEnto} above which is the case $p = 2$.  

\begin{example} Erman calculates that the diagram 
\[ \pi = \pi(0,1,p+1, \ldots, 2p) = \left [
\begin{matrix}
1 & 2 & - & - &\cdots & -\\
\vdots & - & \vdots &  & & \vdots \\
- & - & * & * & \cdots & * 
\end{matrix} \right ].
\]
If $m \pi$ is the Betti diagram of a CM-module then its resolution starts
\[ S^m \vmto{d} S(-1)^{2m} \vpil \cdots .\]
An $a \times (a+m)$ matrix degenerates in codimension $\leq m+1$. 
Since we are considering CM-modules, their codimension
is one less than the length of the degree sequence defining $\pi$,
which is $p+1$. 
So $m\pi$ is in $B_\module$ only if $m \geq p$. 
\end{example}

\begin{example}
The diagram
\[   \beta =  \left [ 
\begin{matrix}
2 & 4 & 3 & - \\
- & 3 & 4 & 2
\end{matrix}
\right ]
\]
is a virtual Betti diagram. Erman shows that $\beta$ and $3 \beta$ 
are not in $B_\module$, but $2 \beta$ is in $B_\module$. In particular
the points on $B_\module$ on a ray in $B_\integ$ may contain nonconsecutive
lattice points.
\end{example}

In \cite{Er2} Erman is able to apply Boij-S\"oderberg theory
to prove the Buchsbaum-Eisenbud-Horrocks conjecture in some
special cases.

\subsubsection{Module theoretic interpretations of the decomposition}
When decomposing the Betti diagram of a module $M$ into a linear
combination of pure diagrams associated to a chain of degree sequences
\begin{equation} \label{FurLigDec}
 \beta(M) = \sum_{i=1}^t c_i \pi(\bfd^i) 
\end{equation}
where $\bfd^1 < \cdots < \bfd^t$, 
one may ask if the decomposition reflects some decomposition of the
the minimal free resolution of $M$. 

Eisenbud, Erman and Schreyer, \cite{EES}, ask if there exists a filtration
of 
\begin{equation} \label{FurLigFilt}
 M = M_t \supset M_{t-1} \supset M_{t-2} \supset \cdots \supset 
M_1 \supset M_0 = 0 
\end{equation}
such that the $M_i /M_{i-1}$ have pure resolutions $c_i \pi(\bfd^i)$. 
Of course in general this cannot be so since the coefficient $c_i$
in the decomposition of $\beta(M)$ may not be integers. 
However even in the case of integer $c_i$, examples show there
may not exist such a filtration of $M$ or even of $M^{\oplus r}$
for $r \geq 1$, \cite[Ex. 4.5]{SW}. There is the question though if 
some deformation or specialization of $M$ or $M^{\oplus r}$ for $r \geq 1$ could
have such a filtration.

In \cite{EES} they give sufficient conditions on chains of degree sequences 
such that if $M$ has a decomposition
(\ref{FurLigDec}), the $c_i$ are integers and there is a filtration
(\ref{FurLigFilt}). 
As a particular striking application they give the following.

\begin{example} Let $S = \kr[x_1, x_2, x_3]$ and $p = 2k+1$
be an odd prime. Consider the pure diagrams
\[ \pi(0,1,2,p) = \left [ 
\begin{matrix}
\binom{p-1}{2} & p^2 - 2p &\binom{p}{2}  & -  \\
\vdots & & & \vdots \\
- & - & - & 1  
\end{matrix}
\right ], \,\,
 \pi(0,p-2,p-1,p) = 
 \left [ 
\begin{matrix}
1 & - & - & -  \\
\vdots & & & \vdots \\
- & \binom{p}{2} & p^2 - 2p & \binom{p-1}{2}   
\end{matrix}
\right ]
\]
and 
\[ \pi(0,k,k+1,p) =  \left [ 
\begin{matrix}
1 & - & -  & -  \\
  & \vdots & &  \\
- & p & p & - \\
 & & \vdots & \\
- & - & - & 1 
\end{matrix}
\right ].
\]
If $\alpha + 1 + \binom{p-1}{2} \equiv 0 \pmod p$ 
the diagram 
\[ \beta = \frac{1}{p} \pi(0,1,2,p) + \frac{\alpha}{p}
\pi(0,k,k+1,p) + \frac{1}{p} \pi(0,p-2,p-1,p) \] is an
integer diagram and the integer
diagrams on the ray of $\beta$ are $m \beta$ where $m$ is a positive integer. 
In \cite{EES} they show that if $M$
is a module with Betti diagram $m\beta$, then it has a filtration
(\ref{FurLigFilt}). 
Thus the coefficients in the decomposition of $m \beta$ are integers
and so $m$ must be divisible by $p$. Hence we have a ray
where only $\frac{1}{p}$ of the lattice points are actual Betti diagrams
of modules.
\end{example}

\subsection{Variations on the grading}
\label{SubsekFurGrad}
The Boij-S\"oderberg conjectures concerns modules over the standard
graded polynomial ring $\kr[x_1, \ldots, x_r]$ where each
$\deg x_i = 1$. 
Since the conjectures have been settled, it is natural to 
consider variations on the grading or finer gradings on the polynomial
ring and module categories. 
Specifically each degree of $x_i$ may be an element of ${\nat}_0^r \backslash
\{ {\mathbf 0} \}$
and the modules $\hele^r$-graded. 
When the module category we consider is closed under direct sums, the
positive rational multiples of their Betti diagrams form a cone.
The topic one has been most interested in concerning other
gradings, is what are the extremal rays in this cone.
This is the analog of the rays of pure
resolutions. More precisely one is interested in finding the
Betti diagrams $\beta(M)$ such that if 
one has a positive linear combination
\[ \beta(M) = q_1 \beta(M_1) + q_2 \beta(M_2),\]
where $M_1$ and $M_2$ are other modules in the category, then
each $\beta(M_i)$ is a multiple of $\beta(M)$. 

In the case $S = \kr [x,y]$ and $\deg x = 1, \deg y = 2$, this is
investigated in the note \cite{BBetal} by B.Barwick et.al.
where they give candidates for
what the extremal rays are, although no proofs.

\medskip
Also with the polynomial ring in two variables, Boij and Fl{\o}ystad, 
\cite{BF}, consider the case when $\deg x = (1,0), \deg y = (0,1)$. 
They fix a degree sequence $(0,p,p+q)$ and consider bigraded artinian
modules whose resolution becomes pure of this type when taking
total degrees by the map $\hele^2 \pil \hele$ given by $(d_1, d_2) \mapsto
d_1 + d_2$. Let $P(p,q)$ be the positive rational cone generated
by such modules.

\begin{theorem}[\cite{BS}] When $p$ and $q$ are relatively prime
the extremal rays in the cone $P(p,q)$ are parametrized by pairs
$(a,I)$ where $a$ is an integer and $I$ is an order ideal (down set)
in the partially ordered set $\nat^2$, contained in the region
$px + qy < (p-1)(q-1)$. 
\end{theorem}

In particular there is a maximal order ideal in this region; it 
corresponds to the equivariant resolution. And there is a minimal
order ideal, the empty set; it corresponds to a resolution of 
a quotient of monomial ideals given in the original \cite[Remark 3.2]{BS}. 

\medskip
For the polynomial ring in any number $r$ of variables, Fl{\o}ystad 
\cite{Fl} lets $\deg x_i$ be the $i$'th unit
vector $e_i$. He considers $\hele^r$-graded artinian modules whose
resolutions becomes pure of a given type $(d_0, d_1, \ldots, d_c)$ when
taking total degrees. He gives a complete description of the linear
space generated by their multigraded Betti diagrams, see Theorem
\ref{ExiTheSchurArtin} in this survey.

\medskip
Instead of Betti diagrams one may consider cohomology tables arising
from other gradings.
Eisenbud and Schreyer in \cite{ES} consider vector bundles $\gF$ on
$\proj{1} \times \proj{1}$. The cohomology groups 
\begin{equation} \label{FurLigKoh}
 H^i \gF(a,b), \quad i = 0,1,2, \quad (a,b) \in \hele^2 
\end{equation} 
give a cohomology
table in $\oplus_{(a,b) \in \hele^2} \rat^3$. 
One gets a positive rational cone of bigraded cohomology tables
and one may ask what are the extremal rays of this cone. 
If for each $(a,b)$ the  cohomology groups (\ref{FurLigKoh}) are
nonvanishing for at most one $i$, $\gF$ is said to have natural cohomology.
In \cite{ES} they give sufficient conditions for a vector bundle
with natural cohomology to be on an extremal ray.

Let us end the subsection 
with a quote by F.-O. Schreyer \cite{Sch}: ``Very little
is known for the extension of this theory to the multi-graded setting.
I believe that there will be beautiful results ahead in this direction.''

\subsection{Poset structures}
\label{SubsekFurPoset}
In the unique decomposition of a Betti diagram 
\begin{equation*} 
 \beta(M) = \sum_{i=1}^s c_i \pi(\bfd^i) 
\end{equation*}
that we consider, we require that the degree sequences form a chain
$\bfd^1 < \bfd^2 < \cdots <\bfd^r$. 

C.Berkesch et.al., \cite{BEKS}, show that this order condition 
is reflected on modules with pure resolutions.

\begin{theorem}[\cite{BEKS}] Let $\bfd$ and $\bfd^\prime$ be
degree sequences. Then $\bfd \leq \bfd^\prime$ if and only if there
exists Cohen-Macaulay modules $M$ and $M^\prime$ with 
pure resolutions of types $\bfd$ and $\bfd^\prime$ with 
a nonzero morphism $M^\prime  \pil M$ of degree $\leq 0$. 
\end{theorem}

They also show the analog of this for vector bundles.
This point of view may be fruitful when trying to understand
decomposition algorithms of Betti diagrams under variations on the
gradings.

\medskip D. Cook, \cite{Co}, investigates the posets $[\bfa, \bfb]_\inc$
and shows that they are vertex-decomposable, Cohen-Macaulay and square-free
glicci.

\subsection{Computer packages}
\label{SubsekFurComp}
Macaulay 2 has the package ``BoijSoederberg''. We mention the most 
important routines in this package.
\begin{itemize}
\item {\bf decompose}: Decomposes a Betti diagram $B$ as a positive
linear combination of pure diagrams.
\item {\bf pureBettiDiagram}: Lists the smallest positive integral 
Betti diagram of a pure resolution of a given type.
\item {\bf pureCohomologyTable}: Gives the smallest positive integral 
cohomology table for a given root sequence.
\item {\bf facetEquation}: Computes the upper facet equation of 
a given facet of Type 3.
\item Routines to compute the Betti numbers for all three
pure resolutions constructed in Section 3.
\begin{itemize}
\item The equivariant resolution.
\item The characteristic free resolution.
\item The resolutions associated to generic matrices.
\end{itemize}
\end{itemize}

The package ``PieriMaps'' contains the routine {\bf PureFree} to 
compute the equivariant resolutions constructed in Subsection 
\ref{SubsekExiEkvi}, and the routine {\bf pieriMaps} to 
compute the more general resolutions of \cite{SW}, see the end of Subsection
\ref{SubsekExiEkvi}.

\subsection{Three basic problems}
\label{SubsekFurPro}
The notes \cite{EMN} is a collection of open questions and problems
related to Boij-S\"oderberg theory. We mention here three problems, 
which we consider to be fundamental. (They are not explicitly in the
notes.)

In \cite{ES} Eisenbud and Schreyer 
give a decomposition of the cohomology table of 
a coherent sheaf on $\proj{n}$ involving an infinite number of data. It does
not seem possible from this to determine the possible cohomology tables
of coherent sheaves up to rational multiple.

\begin{problem} \label{FurProCoh} Determine the possible cohomology tables
of coherent sheaves on $\proj{n}$  up to rational multiple. Can
it by done by essentially a finite number of data?
(At least if you fix a suitable ``window''.)
\end{problem}

Let $E = \oplus_{i = 0}^{\dim V} \wedge^i V$ be the exterior algebra.
A {\it Tate resolution} is an acyclic complex unbounded in each direction
\[ \cdots \pil G^{i-1} \pil G^i \pil G^{i+1} \pil \cdots \]
where each $G^i$ is a free graded $E$-module 
$\oplus_{j \in \hele} E(j)^{\gamma_{i,j}}$.  
To any coherent sheaf $\gF$ is associated a Tate resolution $T(\gF)$, 
see \cite{EFS}. Tate resolutions associated to coherent sheaves 
constitute the class of Tate resolutions which are
eventually linear i.e. such that $G^i = E (i-i_0)$ for $i \gg 0$ and 
some integer $i_0$. 
Hence the following is a generalization of the above Problem \ref{FurProCoh}.

\begin{problem}
Determine the tables $(\gamma_{i,j})$ of Tate resolutions, up to 
rational multiple.
\end{problem}

A complex $\Fd$ of free $S$-modules comes with three natural
sets of invariants: The graded Betti numbers $B$, the Hilbert functions
$H$ of its homology modules, and the Hilbert functions $C$ of the 
homology modules of the dualized complex $D(K)$, where
$D = \Hom(- , \omega_S)$ is the standard duality.

When $H$ and $C$ each live in only one homological degree, $\Fd$ is 
a resolution of a Cohen-Macaulay module and Boij-S\"oderberg theory
describes the positive rational cone Betti of diagrams $B$ 
and, since $H$ and $C$ are determined by $B$, the
set of the triples $(B,H,C)$.
If $H$ only lives in one homological degree, i.e. $\Fd$ is a resolution,
we saw in Subsection \ref{SubsekFurNC}
that Boij and S\"oderberg, \cite{BS2}, gave a description 
of the possible $B$ which are 
projections onto the first coordinate of such triples, up to rational multiple.

\begin{problem}
Describe all triples $(B,H,C)$ that can occur for a complex of
free $S$-modules $\Fd$, up to rational multiple.
Also describe all such triples under
various natural conditions on $B, H$ and $C$. 
\end{problem}


\bibliographystyle{amsplain}
\bibliography{Bibliography}

\end{document}